\documentclass[a4paper,final]{amsart}

\usepackage{showkeys}
\usepackage{fullpage}
\usepackage{mathtools, amssymb, amsthm, amsfonts, mathrsfs}
\usepackage{enumitem}
\usepackage{thmtools}

\usepackage{xcolor, graphicx, here}
\definecolor{darkgreen}{rgb}{0.0, 0.6, 0.0}

\usepackage{tikz}
\usetikzlibrary{intersections, calc, arrows, cd}

\usepackage[colorlinks,citecolor=darkgreen,linkcolor=red]{hyperref}
\usepackage{cleveref}
	
\numberwithin{equation}{section}
\numberwithin{figure}{section}
\allowdisplaybreaks

\setenumerate{label={\upshape (\arabic*)}, nosep}
\setitemize{nosep}
\setdescription{nosep}
\newlist{enua}{enumerate}{1}
\setlist*[enua]{label={\upshape (\arabic*)}, nosep}
\newlist{enur}{enumerate}{1}
\setlist*[enur]{label={\upshape (\roman*)}, nosep}

\theoremstyle{plain}
\newtheorem{thm}{Theorem}[section]
\newtheorem{prp}[thm]{Proposition}
\newtheorem{lem}[thm]{Lemma}
\newtheorem{cor}[thm]{Corollary}
\newtheorem{fct}[thm]{Fact}

\newtheorem{setup}[thm]{Setup}
\newtheorem{cond}[thm]{Condition}
\newtheorem{thma}{Theorem}
\newtheorem{setupa}[thma]{Setup}

\theoremstyle{definition}
\newtheorem{dfn}[thm]{Definition}

\newtheorem{rmk}[thm]{Remark}
\newtheorem{ex}[thm]{Example}

\crefname{thm}{Theorem}{Theorems}
\crefname{thma}{Theorem}{Theorems}
\crefname{prp}{Proposition}{Propositions}
\crefname{lem}{Lemma}{Lemmas}
\crefname{cor}{Corollary}{Corollaries}
\crefname{dfn}{Definition}{Definitions}
\crefname{rmk}{Remark}{Remarks}
\crefname{fct}{Fact}{Facts}
\crefname{ex}{Example}{Examples}
\crefname{ques}{Question}{Questions}
\crefname{setup}{Setup}{Setups}
\crefname{setupa}{Setup}{Setups}
\crefname{cond}{Condition}{Conditions}



\DeclareMathOperator{\dgMod}{\mathsf{dgMod}}
\DeclareMathOperator{\dgmod}{\mathsf{dgmod}}
\DeclareMathOperator{\Qcoh}{\mathsf{Qcoh}}

\DeclareMathOperator{\Ob}{\mathsf{Ob}}
\DeclareMathOperator{\Mor}{\mathsf{Mor}}

\DeclareMathOperator{\Filt}{\mathsf{Filt}}

\DeclareMathOperator{\Pic}{Pic}
\DeclareMathOperator{\LSpec}{LSpec}
\DeclareMathOperator{\LSupp}{LSupp}
\DeclareMathOperator{\LAss}{LAss}

\DeclareMathOperator{\noeth}{noeth}
\DeclareMathOperator{\fp}{fp}

\DeclareMathOperator{\Cone}{Cone}

\DeclareMathOperator{\dgHom}{\mathcal{H}\textit{\kern -1.5pt om}\hspace{0.5pt}}

\DeclareMathOperator{\torf}{\mathsf{torf}}
\DeclareMathOperator{\tors}{\mathsf{tors}}
\DeclareMathOperator{\serre}{\mathsf{serre}}

\DeclareMathOperator{\Loc}{\mathsf{Loc}}

\newcommand{\imply}{\Rightarrow}

\newcommand{\Equi}{\quad \Longleftrightarrow \quad}

\newcommand{\wti}{\widetilde}

\newcommand{\ol}{\overline}

\newcommand{\xr}[1]{\xrightarrow{\, #1 \, }}
\newcommand{\inj}{\hookrightarrow}
\newcommand{\surj}{\twoheadrightarrow}

\newcommand{\isoto}{\xr{\simeq}}

\newcommand{\iso}{\cong}

\newcommand{\bbZ}{\mathbb{Z}}

\newcommand{\calL}{\mathcal{L}}

\newcommand{\bbk}{\Bbbk}

\newcommand{\catA}{\mathcal{A}}
\newcommand{\catB}{\mathcal{B}}
\newcommand{\catC}{\mathcal{C}}

\newcommand{\catF}{\mathcal{F}}

\newcommand{\catS}{\mathcal{S}}
\newcommand{\catT}{\mathcal{T}}

\newcommand{\catX}{\mathcal{X}}
\newcommand{\catY}{\mathcal{Y}}

\DeclareMathOperator{\Mod}{\mathsf{Mod}}

\DeclareMathOperator{\catmod}{\mathsf{mod}}

\DeclareMathOperator{\coh}{\mathsf{coh}}

\newcommand{\id}{\mathrm{id}}
\DeclareMathOperator{\Hom}{Hom}

\DeclareMathOperator{\Ima}{Im}
\DeclareMathOperator{\Ker}{Ker}
\DeclareMathOperator{\Cok}{Cok}

\newcommand{\pb}{\arrow[rd,"{\mathrm{PB}}",phantom]}

\newcommand{\op}{\text{op}}

\DeclareMathOperator{\Spec}{Spec}

\DeclareMathOperator{\Ass}{Ass}

\newcommand{\mm}{\mathfrak{m}}

\newcommand{\pp}{\mathfrak{p}}

\newcommand{\kk}{\kappa}

\DeclareMathOperator{\Supp}{Supp}

\newcommand{\shF}{\mathscr{F}}

\newcommand{\shL}{\mathscr{L}}

\newcommand{\shO}{\mathscr{O}}

\DeclareMathOperator{\shHom}{\mathscr{H}\textit{\kern -3pt om}\hspace{0.5pt}}

\newcommand{\la}{\langle}
\newcommand{\ra}{\rangle}

\makeatletter

\newcommand{\colim@}[2]{%
  \vtop{\m@th\ialign{##\cr
    \hfil$#1\operator@font colim$\hfil\cr
    \noalign{\nointerlineskip\kern1.5\ex@}#2\cr
    \noalign{\nointerlineskip\kern-\ex@}\cr}}%
}

\newcommand{\dlim}{%
  \mathop{\mathpalette\colim@{\rightarrowfill@\scriptscriptstyle}}\nmlimits@
}

\newcommand{\plim}{%
  \mathop{\mathpalette\varlim@{\leftarrowfill@\scriptscriptstyle}}\nmlimits@
}

\makeatother

\begin{document}
\title{Atom spectra of symmetric monoidal abelian categories and classification of subcategories}
\author{Shunya Saito}
\address{Graduate School of Mathematical Sciences, The University of Tokyo, 3-8-1 Komaba, Meguro-ku, Tokyo 153-8914, Japan}
\email{shunya-saito@g.ecc.u-tokyo.ac.jp}


\begin{abstract}
We extend the classification results for torsion classes and torsion-free classes in the category of finitely generated modules over a commutative noetherian ring
to suitable symmetric monoidal closed noetherian abelian categories.

Our main tool is the orbit atom spectrum,
defined as the quotient of Kanda's atom spectrum by the action induced by tensoring with invertible objects.
We prove that, under natural tensor-theoretic assumptions,
several classes of subcategories collapse to Serre subcategories or torsion-free classes.
Moreover, torsion-free classes compatible with the tensor structure are classified by arbitrary subsets of the orbit atom spectrum.

As applications, we recover the classical classifications for commutative noetherian rings
and obtain analogues for graded modules, coherent sheaves, and dg modules.
\end{abstract}

\maketitle
\tableofcontents

\section{Introduction}\label{s:intro}
\subsection{Background}
The classification of subcategories of abelian categories is a classical problem in representation theory, algebraic geometry, and commutative algebra.
Many natural classes of subcategories encode structural features of abelian categories.  
For example, Serre subcategories control the localization theory of abelian categories.
Torsion classes and torsion-free classes arise as the two halves of torsion pairs,
and hence describe ways of decomposing objects in an abelian category.
Wide subcategories are extension-closed subcategories which are abelian in their own right.
These subcategories are all defined by being closed under basic operations in an abelian category, such as extensions, subobjects, quotients, kernels, cokernels, and images.
Thus, given an abelian category, it is a basic problem to understand its full subcategories satisfying such natural closure conditions.

One of the most classical classification theorems is due to Gabriel \cite{Gab}:
for a noetherian scheme $X$, localizing subcategories of the category $\Qcoh X$ of quasi-coherent sheaves are classified in terms of specialization-closed subsets of the underlying topological space of $X$.
In particular, for a commutative noetherian ring $R$, Serre subcategories of the category $\catmod R$ of finitely generated $R$-modules are classified by specialization-closed subsets of
$\Spec R$.
This result has served as a prototype for many later classification theorems.

Another fundamental classification result is due to Takahashi \cite{Takahashi}:
for a commutative noetherian ring $R$, torsion-free classes of $\catmod R$ are classified by arbitrary subsets of $\Spec R$.
Subsequent results showed that, over a commutative noetherian ring, many apparently different closure conditions collapse to one of two types. 
More precisely, torsion classes, wide subcategories, ICE-closed subcategories, and
CE-closed subcategories coincide with Serre subcategories \cite{Takahashi,SW};
similarly, IKE-closed subcategories and IE-closed subcategories coincide with torsion-free classes \cite{IMST,Eno}.
For the precise definitions of the subcategories appearing above, see \cref{dfn:several subcat}.
See also \cite{Krause-torf,CS,IK,KS,Sai-coh,KS-2} for related developments and extensions.

These results reveal a remarkable rigidity in the commutative noetherian case.
In many other settings, the subcategories mentioned above usually do not coincide and exhibit much richer behavior.
For example, over finite-dimensional algebras, such subcategories are closely related to tilting theory, $\tau$-tilting theory, and silting theory.
Through these theories, they are connected with noncrossing partitions, cluster combinatorics, and other combinatorial structures \cite{IT,AIR,MS,Asai,Eno-ICE,ES-ICE}.
In contrast, many natural closure conditions on subcategories of $\catmod R$ collapse into only two familiar types: Serre subcategories and torsion-free classes.
Moreover, both types are controlled by the same topological space $\Spec R$:
Serre subcategories correspond to specialization-closed subsets, and torsion-free classes correspond to arbitrary subsets.

The purpose of this paper is to show that these phenomena are not special to commutative rings.
Rather, once the tensor product is taken into account, they are instances of a general phenomenon in symmetric monoidal abelian categories.

\subsection{Main results}
We now explain our main results.
The basic setting of this paper is a \emph{locally noetherian Grothendieck cosmos},
that is, a locally noetherian Grothendieck category $\catA$ equipped with a symmetric monoidal closed structure.
Thus, it is equipped with a tensor product, a unit object $I$, and internal Hom objects $[M,N]$, characterized by the adjunction between $-\otimes M$ and $[M,-]$.
In what follows, let $\catA$ be a locally noetherian Grothendieck cosmos.
We write $\catA_{\noeth}$ for the full subcategory of noetherian objects.
The basic examples of locally noetherian Grothendieck cosmoi include the following categories:
\begin{itemize}
\item the category $\Mod R$ of modules over a commutative noetherian ring,
\item the category $\Mod^G S$ of graded modules over a commutative graded noetherian ring,
\item the category $\Qcoh X$ of quasi-coherent sheaves on a noetherian scheme,
\item the category $\dgMod A$ of dg modules over a skew-commutative dg noetherian ring.
\end{itemize}

In an abelian category equipped with a symmetric monoidal structure, it is natural to study subcategories which are compatible with the tensor product.
A subcategory $\catX$ is called a \emph{tensor ideal}
if $X\otimes M\in \catX$ for all $X\in\catX$ and all objects $M\in\catA$.
Tensor ideals are especially natural from the viewpoint of localization.
A localizing subcategory which is also a tensor ideal is called a \emph{tensor localizing subcategory}.
If a localizing subcategory $\catX$ is a tensor ideal,
then the quotient category $\catA/\catX$ again inherits a symmetric monoidal closed structure, and the quotient functor is monoidal.
Thus, tensor localizing subcategories provide the localization theory compatible with the tensor product.

However, the condition of being a tensor ideal involves tensoring with all objects of $\catA$,
and is therefore difficult to control directly.
There is a weaker and more tractable condition obtained by considering only invertible objects.
Recall that an object $L$ is \emph{invertible} if the functor $-\otimes L$ is an autoequivalence.
We say that a subcategory $\catX$ of $\catA$ is \emph{$L$-closed}
if $X\otimes L\in \catX$ for all $X\in\catX$ and all invertible objects $L$ of $\catA$.
This condition is natural in many examples.
An additive subcategory of $\catmod R$ closed under direct summands is automatically L-closed,
whereas an additive subcategory of $\catmod^G S$ closed under direct summands is L-closed
if and only if it is closed under grading shifts.

We prove our main results under the following setup.
A key consequence of this setup is that, although being L-closed is a priori weaker,
it is equivalent to being a tensor ideal for the relevant classes of subcategories.
\begin{setupa}[{cf.\ \cref{setup:classify subcat}}]\label{setup:A}
For a locally noetherian Grothendieck cosmos $\catA$, assume that:
\begin{itemize}
\item
The unit object $I$ is finitely presented.
\item
The Grothendieck category $\catA$ is generated by a small set of invertible-filtered objects.
\end{itemize}
\end{setupa}
Here, an object is called \emph{invertible-filtered}
if it is obtained by a finite iterated extension of invertible objects.
The conditions of \cref{setup:A} are satisfied by $\Mod R$, $\Mod^G S$, $\dgMod A$, and $\Qcoh X$ for a suitable noetherian scheme. 

We now state the first main result.
\begin{thma}[{cf.\ \cref{prp:Serre=CE}}]\label{thm:B}
Let $\catA$ be a locally noetherian Grothendieck cosmos satisfying the assumptions in \cref{setup:A}.
Let $\catX$ be an L-closed subcategory in $\catA_{\noeth}$.
\begin{enua}
\item
The following are equivalent in $\catA_{\noeth}$:
\begin{itemize}
\item $\catX$ is a Serre subcategory.
\item $\catX$ is a torsion class.
\item $\catX$ is a wide subcategory.
\item $\catX$ is an ICE-closed subcategory.
\item $\catX$ is a CE-closed subcategory.
\end{itemize}
\item
The following are equivalent in $\catA_{\noeth}$:
\begin{itemize}
\item $\catX$ is a torsion-free class.
\item $\catX$ is an IKE-closed subcategory.
\item $\catX$ is an IE-closed subcategory.
\end{itemize}
\end{enua}
\end{thma}
Thus, under \cref{setup:A}, among L-closed subcategories,
torsion classes, wide subcategories, ICE-closed subcategories, and CE-closed subcategories are precisely Serre subcategories.
Likewise, IKE-closed subcategories and IE-closed subcategories are precisely torsion-free classes.

The proof of \cref{thm:B} is based on the theory of atom spectra developed by Kanda \cite{Kan-Serre}.
For an abelian category $\catA$, the atom spectrum is a topological space $\Spec \catA$ that
generalizes the prime spectrum of a commutative ring.
A fundamental result in the theory of atom spectra says that,
when $\catA$ is a locally noetherian Grothendieck category, localizing subcategories of $\catA$ are classified by open subsets of $\Spec \catA$.
Thus $\Spec \catA$ controls the ordinary localization theory of $\catA$.

In the symmetric monoidal case, however, the relevant localizations are those compatible with the tensor product.
Recall that the Picard group $\Pic \catA$ is the group of isomorphism classes of invertible objects of $\catA$ under the tensor product.
It acts on $\catA$ by tensoring with invertible objects, and hence on $\Spec \catA$.
We therefore define the \emph{orbit atom spectrum} by
\[
\LSpec \catA := \Spec \catA / \Pic \catA.
\]
As in Kanda's theory, one can define subsets $\LSupp M$ and $\LAss M$ of $\LSpec \catA$ for each object $M \in \catA$, analogous to the support and associated primes of a module.

The ordinary atom spectrum $\Spec\catA$ controls ordinary abelian localizations,
whereas the orbit atom spectrum $\LSpec\catA$ controls monoidal abelian localizations.
More precisely, under \cref{setup:A}, tensor localizing subcategories are exactly L-closed localizing subcategories.
Hence, tensor localizations are classified by open subsets of $\LSpec\catA$.
Using this, we can define the localization of $\catA$ at a point $x \in \LSpec\catA$,
which produces a locally noetherian Grothendieck cosmos $\catA_x$ that behaves like a local category in the following sense: its simple objects are unique up to tensoring with invertible objects.
This local nature is the key ingredient of the proof of \cref{thm:B}.

By \cref{thm:B}, under \cref{setup:A}, the subcategories of $\catA_{\noeth}$ considered above collapse to two types:
Serre subcategories and torsion-free classes.
The former are already classified by the orbit atom spectrum:
L-closed Serre subcategories of $\catA_{\noeth}$ correspond to open subsets of $\LSpec\catA$.
Thus, it remains to classify the L-closed torsion-free classes.

For this, we impose one additional assumption,
which is an analogue of the familiar relation between support and associated primes
over a commutative noetherian ring:
every prime in the support of a module contains an associated prime.
\begin{thma}[{cf.\ \cref{prp:classify G-torf}}]\label{thm:C}
Let $\catA$ be a locally noetherian Grothendieck cosmos satisfying \cref{setup:A}.
In addition, assume the following condition:
\begin{itemize}
\item
For any $M\in \catA$ and $x\in \LSupp M$, there exists $y\in \LAss M$ such that $y$ belongs to the topological closure of $\{x\}$ in $\LSpec \catA$.
\end{itemize}
Then there is an order-preserving bijection between the following sets:
\begin{enur}
\item
the set of $L$-closed torsion-free classes of $\catA_{\noeth}$.
\item
the power set of $\LSpec \catA$.
\end{enur}
The bijection is given by
\[
\Phi \mapsto \LAss^{-1}(\Phi):=\{M\in \catA_{\noeth}\mid \LAss(M) \subseteq \Phi\},\quad
\catX \mapsto \LAss \catX:=\bigcup_{X\in \catX} \LAss(X),
\]
where $\Phi$ is a subset of $\LSpec \catA$ and $\catX$ is an $L$-closed torsion-free class.
\end{thma}

\cref{thm:B,thm:C} generalize the classical classifications of subcategories of $\catmod R$ \cite{Takahashi,SW,IMST,Eno}.
However, the proofs are new even in the case of commutative noetherian rings.
In the classical commutative setting, one can use a lifting argument:
an $R_\pp$-linear map $M_{\pp}\to N_{\pp}$ can be lifted to an $R$-linear map $M\to N$, up to a unit of $R_\pp$.
This argument is not well suited to the monoidal abelian setting considered here.
Our proofs avoid it and instead use only global constructions available in a symmetric monoidal closed category, namely the tensor product and the internal Hom.
Thus, our proofs give new conceptual proofs of the classical classifications.

In concrete examples, the orbit atom spectrum is often computable and recovers familiar spaces:
the ordinary prime spectrum for commutative noetherian rings,
the homogeneous prime spectrum for commutative graded noetherian rings,
and the underlying topological space for noetherian schemes.
These computations allow \cref{thm:B} and \cref{thm:C} to recover the known classifications for $\catmod R$ and to produce graded and geometric analogues.

In the dg setting, we obtain a new computation of atom spectra.
For a bounded non-positive right dg noetherian dg ring, not necessarily commutative,
we determine the atom spectrum of the category of dg modules itself in terms of the atom spectrum of the degree-zero part.
In particular, in the skew-commutative case, the relevant orbit atom spectrum is isomorphic to the ordinary prime spectrum of the degree-zero part.
This yields dg analogues of the classification results above.

We also discuss examples outside the scope of \cref{setup:A} in \S \ref{ss:non ex}.
For representations of finite acyclic quivers with the pointwise tensor product and for modules over cocommutative Hopf algebras,
tensor torsion classes are still Serre subcategories, although these categories are not generated by invertible-filtered objects in general.
These examples suggest that the equality between tensor torsion classes and
tensor Serre subcategories should hold beyond the framework of \cref{setup:A}.

However, to prove such a general statement, one may need a different kind of spectrum.
The orbit atom spectrum used in this paper is constructed as a quotient of the atom spectrum by the autoequivalences induced by invertible objects, and is well suited to the framework of \cref{setup:A}.
To treat more general symmetric monoidal abelian categories,
it seems more natural to construct the theory directly from objects such as ``prime tensor localizing subcategories'', as in Balmer's tensor triangular geometry \cite{Balmer}.
Even if such a theory were developed, the results of this paper would still be of interest,
since orbit atom spectra are often computable in concrete examples.

\medskip
\noindent
\textbf{Organization.}
This paper is organized as follows.
In \S \ref{s:pre},
we review basic facts on atom spectra and symmetric monoidal categories.
In \S \ref{s:Atom with G-act},
for an abelian category $\catA$ equipped with an action of a group $G$, 
we introduce the \emph{orbit atom spectrum} $\Spec^G \catA$, defined as the quotient space of the atom spectrum $\Spec \catA$ by $G$.
For a symmetric monoidal abelian category, if we take $G=\Pic\catA$,
then $\Spec^G\catA = \LSpec \catA$, which was introduced in the introduction.
We also investigate basic properties of $\Spec^G \catA$ and the localization of $\catA$ at a point of $\Spec^G \catA$.
In \S \ref{s:LSpec},
we apply the theory of orbit atom spectra developed in \S \ref{s:Atom with G-act} to symmetric monoidal abelian categories.
We show that, under suitable assumptions, monoidal abelian localizations can be controlled by the orbit atom spectrum, and that the localization at a point of $\Spec^G \catA$ behaves well with respect to the monoidal structure.
In \S \ref{s:classify subcat},
we prove \cref{thm:B,thm:C} by applying the theory developed in \S \ref{s:Atom with G-act} and \S \ref{s:LSpec}.
In \S \ref{s:ex}, we describe the orbit atom spectrum for several concrete Grothendieck cosmoi satisfying \cref{setup:A} and discuss consequences of \cref{thm:B,thm:C}.
In \S \ref{ss:non ex}, we give examples of Grothendieck cosmoi that lie outside the scope of \cref{setup:A} but nevertheless have the property that tensor torsion classes are Serre subcategories.


\medskip
\noindent
{\bf Conventions.}
Throughout this paper,
we fix a Grothendieck universe.
A set is said to be \emph{small} if it belongs to the Grothendieck universe.
For a category $\catC$,
the collection $\Ob\catC$ of objects and the collection $\Mor \catC$ of morphisms are sets.
For any objects $M,N \in \catC$,
the set ${\catC}(M,N)$ of morphisms from $M$ to $N$ is assumed to be small.
All rings, modules, schemes, and so forth are assumed to be small.
A category is said to be \emph{essentially small}
if there exists a bijection between the set of isomorphism classes of objects and a small set.

All subcategories are supposed to be full subcategories closed under isomorphisms.
Thus, we often identify the subcategories with the subsets of the set of isomorphism classes
of objects.

\medskip
\noindent
{\bf Acknowledgement.}
The author would like to thank Ryo Kanda and Yuki Imamura for their helpful discussions.
The author is supported by JSPS KAKENHI Grant Number JP24KJ0057.

\section{Preliminaries}\label{s:pre}
\subsection{Basic concepts in abelian categories}\label{ss:abel cat}
In this subsection, we briefly recall some basic concepts of abelian categories in order to fix notation and terminology.
Throughout this subsection, let $\catA$ be an abelian category.
\begin{dfn}\label{dfn:several subcat}
Let $\catX$ be an additive subcategory of $\catA$.
\begin{enua}
\item
$\catX$ is said to be \emph{closed under coproducts}
if, for any family $\{X_i\}_i$ of objects in $\catX$,
whenever the coproduct $\bigoplus_i X_i$ exists in $\catA$, we have $\bigoplus_i X_i \in \catX$.
Similarly, we define \emph{closed under products}.
\item
$\catX$ is said to be \emph{extension-closed} (or \emph{closed under extensions})
if for any exact sequence $0 \to A \to B \to C \to 0$,
we have that $A,C \in \catX$ implies $B\in \catX$.
\item
$\catX$ is said to be \emph{closed under images} (resp.\ \emph{kernels}, resp.\ \emph{cokernels})
if for any morphism $f \colon X \to Y$ in $\catA$ with $X,Y \in \catX$,
we have that $\Ima f \in \catX$ (resp.\ $\Ker f \in \catX$, resp.\ $\Cok f \in \catX$).
\item
$\catX$ is said to be \emph{closed under subobjects} (resp.\ \emph{quotients})
if for any injection $A \inj X$ (resp.\ surjection $X \surj A$) in $\catA$ such that $X\in \catX$,
we have that $A \in \catX$.
\item
$\catX$ is called a \emph{torsion-free class}
if it is closed under subobjects and extensions.
\item
$\catX$ is called a \emph{torsion class}
if it is closed under quotients and extensions.
\item
$\catX$ is called a \emph{Serre subcategory} 
if it is closed under subobjects, quotients and extensions.
\item
$\catX$ is said to be \emph{wide}
if it is closed under kernels, cokernels and extensions.
\item
$\catX$ is said to be \emph{IE-closed}
if it is closed under images and extensions.
\item
$\catX$ is said to be \emph{IKE-closed}
if it is closed under images, kernels and extensions.
\item
$\catX$ is said to be \emph{ICE-closed}
if it is closed under images, cokernels and extensions.
\item
$\catX$ is said to be \emph{CE-closed}
if it is closed under cokernels and extensions.
\item
$\catX$ is said to be \emph{KE-closed}
if it is closed under kernels and extensions.

\end{enua}
\end{dfn}

Let $\catX$ be a set of objects of $\catA$.
We denote by $\catX^{\perp}$ (resp.\ ${}^{\perp}\catX$) the subcategory of $\catA$ consisting of objects $M$ such that $\catA(X,M)=0$ (resp.\ $\catA(M,X)=0$) for any $X\in \catX$.
An object $M$ is \emph{$\catX$-filtered} if there exists a finite filtration
\[
0=M_0 \subseteq M_1 \subseteq \cdots \subseteq M_n=M
\]
such that $X_i:=M_i/M_{i-1} \in \catX$ for any $i$.
In this case, we write $M \in X_1*\cdots *X_n$.
We denote by $\Filt \catX$ the subcategory of $\catA$ consisting of $\catX$-filtered objects.

A pair $(\catT,\catF)$ of subcategories of $\catA$ is called a \emph{torsion pair}
if $\Hom_{\catA}(\catT,\catF)=0$ and for any $M\in \catA$, there exists an exact sequence
$0 \to T \to M \to F \to 0$ such that $T \in \catT$ and $F\in\catF$.
In this case, $\catT^{\perp}=\catF$ and ${}^{\perp}\catF=\catT$ hold.
Hence, the subcategory $\catT$ is a torsion class closed under coproducts and $\catF$ is a torsion-free class closed under products.
Conversely, if $\catA$ is noetherian or $\catA$ has small coproducts,
then $(\catX,\catX^{\perp})$ is a torsion pair if and only if $\catX$ is a torsion class closed under coproducts.
A torsion pair $(\catT,\catF)$ is said to be \emph{hereditary}
if $\catT$ is closed under subobjects, that is, it is a Serre subcategory.

Let $\catA$ be a Grothendieck category.
See \cite{Ste} for a detailed account of Grothendieck categories.
The subcategory of $\catA$ consisting of noetherian objects (resp.\ finitely presented objects) is denoted by $\catA_{\noeth}$ (resp.\ $\catA_{\fp}$).
If $\catA$ is locally noetherian, then $\catA_{\noeth}=\catA_{\fp}$.
A \emph{localizing subcategory} is a Serre subcategory closed under coproducts. 
If $\catX$ is localizing, the quotient functor $Q\colon \catA \to \catA/\catX$ has a right adjoint.
Then $\catX$ and $\catA/\catX$ are also Grothendieck categories.
Moreover, if $\{G_i\}_{i\in I}$ is a small generating set of $\catA$,
then $\{Q(G_i)\}_{i\in I}$ is a small generating set of $\catA/\catX$.
If $\catA$ is locally noetherian, then $\catX$ and $\catA/\catX$ are also locally noetherian.
In this case, $(\catA/\catX)_{\noeth}=\catA_{\noeth}/(\catX\cap\catA_{\noeth})$ holds.

In a Grothendieck category, every nonzero finitely generated object has a maximal subobject (see \cite[Lemma 2.5]{Kra-KS} for example).
From this, we obtain the following observation.
\begin{lem}\label{prp:Nakayama}
Let $\catA$ be a Grothendieck category and $M \in \catA$ a finitely generated object.
If $\catA(M,S)=0$ for any simple object $S$,
then $M=0$.\qed
\end{lem}
\begin{lem}
Every nonzero locally noetherian Grothendieck category has a simple object.\qed
\end{lem}
\subsection{Atom spectra}
In this subsection, we review the definition and basic results of atom spectra following Kanda's works \cite{Kan-Serre,Kan-spcl,Kan-CatSp}.
Throughout this subsection, let $\catA$ be an abelian category.

An object $H\in \catA$ is said to be \emph{monoform}
if for any nonzero subobject $L$ of $H$,
there exists no common nonzero subobject of $H$ and $H/L$.
Two monoform objects $H_1$ and $H_2$ in $\catA$ are \emph{atom-equivalent}
if there exists a common nonzero subobject of $H_1$ and $H_2$.
Denote by $\Spec \catA$ the set of atom-equivalence classes of monoform objects,
and call it the \emph{atom spectrum} of $\catA$.
The \emph{atom support} of $M\in \catA$ is defined by
\[
\Supp_{\catA} M := \{\ol{H} \in \Spec \catA \mid \text{$H$ is subquotient of $M$}\}.
\]
We often write $\Supp M := \Supp_{\catA} M$ for simplicity.
For $M\in \catA$ and a monoform object $H$,
note that $\ol{H} \in \Supp M$ if and only if there is a subobject $N$ of $M$
such that $M/N$ and $H$ have a common nonzero subobject.
The atom spectrum $\Spec \catA$ carries a topology whose open basis consists of $\{\Supp M \mid M\in\catA\}$.

An element of the following set is called an \emph{associated atom} of $M\in\catA$:
\[
\Ass M := \{\ol{H} \in \Spec \catA \mid \text{$H$ is a subobject of $M$}\}.
\]
We often write $\Ass M := \Ass_{\catA} M$ for simplicity.
If $M$ is nonzero noetherian object, then $\Ass M$ is a non-empty finite set.
The following observation is useful to study associated atoms.
\begin{lem}\label{prp:monoform sub lem}
Let $\Phi$ be an open subset of $\Spec \catA$, and let $M \in \catA$.
Then for any $\alpha \in \Ass M \cap \Phi$,
there is a monoform subobject $H$ of $M$
such that $\ol{H}=\alpha$ and $\Supp H \subseteq \Phi$.
\end{lem}
\begin{proof}
By the definition of associated atoms,
there is a monoform subobject $H'$ of $M$ such that $\alpha = \ol{H'}$.
Since $\Phi$ is open and $\alpha \in \Phi$,
there is also a monoform object $H''$ such that $\Supp H'' \subseteq \Phi$ and $\ol{H'}=\alpha=\ol{H''}$ (see \cite[Definition 3.7]{Kan-Serre}).
Then we have a nonzero common subobject $H$ of $H'$ and $H''$,
and $H$ is a desired object.
\end{proof}

For any subset $\Phi$ of $\Spec \catA$, define subcategories of $\catA$ by
\[
\Supp^{-1}(\Phi):=\{M\in \catA \mid \Supp M \subseteq \Phi\},\quad
\Ass^{-1}(\Phi) :=\{M\in \catA \mid \Ass M \subseteq \Phi\}.
\]
Then $\Supp^{-1}(\Phi)$ is a Serre subcategory and $\Ass^{-1}(\Phi)$ is a torsion-free class.

Let $\catA$ be a locally noetherian Grothendieck category.
Then the inclusion $\catA_{\noeth} \inj \catA$ induces a homeomorphism
$\Spec \catA_{\noeth} \isoto \Spec \catA$ (cf.\ \cite[Proposition 5.3]{Kan-Serre}).
Under this identification, for a subset $\Phi$ of $\Spec \catA$,
we have
\[
\Supp^{-1}_{\catA}(\Phi)\cap \catA_{\noeth}= \Supp^{-1}_{\catA_{\noeth}}(\Phi),\quad
\Ass^{-1}_{\catA}(\Phi)\cap \catA_{\noeth}= \Ass^{-1}_{\catA_{\noeth}}(\Phi).
\]
The following fundamental result shows that localizing subcategories are controlled by the atom spectrum.
\begin{fct}[{\cite[Theorem 5.5]{Kan-Serre}}]\label{fct:Kanda classify Serre}
Let $\catA$ be a locally noetherian Grothendieck category.
Then there exist bijections among the following sets:
\begin{enur}
\item 
the set of localizing subcategories of $\catA$.
\item 
the set of Serre subcategories of $\catA_{\noeth}$.
\item
the set of open subsets of $\Spec \catA$.
\end{enur}
These bijections are given as follows:
\begin{itemize}
\item
The bijection between (i) and (ii) is given by $\catX \mapsto \catX \cap \catA_{\noeth}$ and $\catS \mapsto \la \catS \ra_{\mathrm{loc}}$,
where $\catX$ is a localizing subcategory of $\catA$,
$\catS$ is a Serre subcategory of $\catA_{\noeth}$,
and $\la \catS \ra_{\mathrm{loc}}$ is the smallest localizing subcategory of $\catA$ containing $\catS$.
\item
The bijection between (ii) and (iii) is given by $\catX \mapsto \Supp \catX:=\bigcup_{X\in\catX}\Supp X$ and $\Phi \mapsto \Supp_{\catA_{\noeth}}^{-1}\Phi$,
where $\catX$ is a Serre subcategory of $\catA_{\noeth}$ and $\Phi$ is an open subset of $\Spec \catA$.
\end{itemize}
\end{fct}

\begin{rmk}
Let $\catA$ be a locally noetherian Grothendieck category.
The atom spectrum is closely related to the classical Gabriel spectrum.
Indeed, the Gabriel spectrum of $\catA$ is the set of isomorphism classes of indecomposable injective objects.
The injective hull of a monoform object is an indecomposable injective object, and this construction gives a homeomorphism from $\Spec\mathcal A$ to the Gabriel spectrum,
equivalently to the Ziegler spectrum (see \cite[Section 5]{Kan-Serre}).
Under this identification, the classification of localizing subcategories in \cref{fct:Kanda classify Serre} recovers the corresponding classification by Herzog and Krause \cite{Herzog-Spec, Krause-Spec}.

The advantage of using the atom spectrum here is that it is defined purely in terms of monoform objects.
In particular, it can be defined for noetherian abelian categories without referring to injective objects.
Moreover, monoform objects may be viewed as categorical analogues of the $R$-modules $R/\pp$,
where $\pp$ is a prime ideal of a commutative noetherian ring $R$.
This point of view allows us to use the atom spectrum with the same intuition as the prime spectrum.
\end{rmk}

\begin{prp}[{cf.\ \cite{AS}}]\label{prp:hered tors by atom}
Let $\catA$ be a locally noetherian Grothendieck category.
For an open subset $\Phi$ of $\Spec \catA$,
put $\catX_{\Phi}:=\Supp^{-1}(\Phi)$ and $\catY_{\Phi}:=\Ass^{-1}(\Spec \catA \setminus \Phi)$.
Then $(\catX_\Phi, \catY_\Phi)$ is a hereditary torsion pair in $\catA$.
\end{prp}
\begin{proof}
Since $\catX_{\Phi}$ is a Serre subcategory closed under coproducts,
$(\catX_\Phi, (\catX_\Phi)^{\perp})$ is a hereditary torsion pair.
It is enough to show that $(\catX_\Phi)^{\perp}= \catY_\Phi$.
Take a morphism $f\colon X \to Y$ such that $X \in \catX_\Phi$ and $Y \in \catY_\Phi$.
Then $\Ima f \in \catX_\Phi \cap \catY_\Phi$, and thus $\Ass(\Ima f)$ is empty.
As $\catA$ is locally noetherian, this implies $\Ima f =0$.
Hence, we have $\catY_\Phi \subseteq (\catX_\Phi)^{\perp}$.
Conversely, suppose that $M \not\in \catY_\Phi$.
Then $\Ass M \cap \Phi$ is nonempty, and take $\alpha \in \Ass M \cap \Phi$.
There is a monoform subobject $H$ of $M$ such that $\ol{H}=\alpha$ and $\Supp H \subseteq \Phi$.
As $H \in \catX_{\Phi}$, we have $M \not\in (\catX_\Phi)^{\perp}$.
Therefore, we obtain the assertion. 
\end{proof}

The behavior of atom spectra under localization is as follows.
\begin{fct}[{\cite[Theorem 5.4 and Proposition 5.6]{Kan-CatSp}}]\label{fct:Spec of quot}
Let $\catA$ be a Grothendieck category, and let $\catX$ be a localizing subcategory of $\catA$.
Let $Q\colon \catA \to \catA/\catX$ denote the quotient functor,
and $R\colon \catA/\catX \to \catA$ its right adjoint.
Then the assignments $\ol{H} \mapsto \ol{Q(H)}$ and $\ol{H'} \mapsto \ol{R(H')}$
yield a homeomorphism $\Spec (\catA/\catX) \iso \Spec \catA \setminus \Supp_{\catA}\catX$.
Under this identification, the following hold:
\begin{enua}
\item
For every $M \in \catA$, we have
\[
\Ass_{(\catA/\catX)} Q(M) \supseteq \Ass_\catA M \setminus \Supp_{\catA} \catX,\quad
\Supp_{(\catA/\catX)} Q(M) = \Supp_{\catA} M \setminus \Supp_{\catA} \catX.
\]
\item
For every $N \in \catA/\catX$, we have
\[
\Ass_{(\catA/\catX)} N = \Ass_{\catA} R(N),\quad
\Supp_{(\catA/\catX)} N = \Supp_{\catA} R(N) \setminus \Supp_{\catA} \catX.
\]
\end{enua}
\end{fct}


\subsection{Symmetric monoidal categories}
In this subsection, we review some basic facts on symmetric monoidal categories.
For details, see \cite{Kelly,CWM}.
Throughout this subsection, let $\catA$ be a symmetric monoidal closed category.
We write $\otimes_{\catA}$ for the tensor functor and $I_{\catA}$ for the unit object.
For objects $M,N\in \catA$, the internal Hom from $M$ to $N$ is denoted by $\catA[M,N]$.
If there is no ambiguity, we simply write $\otimes$, $I$, and $[M,N]$.
As $(-\otimes M,[M,-])$ is an adjoint pair,
the functor $-\otimes M$ preserves colimits and $[M,-]$ preserves limits.

An object $D \in \catA$ is \emph{dualizable} if there is $E\in \catA$ such that $-\otimes E$ is a right adjoint of $-\otimes D$.
In this case, $E$ is called the \emph{dual} of $D$.
The dual $E$ is isomorphic to $D^{\vee}:=[D,I]$ by the following:
\[
\catA(-,[D,I])
\iso \catA(-\otimes D,I)
\iso \catA(-,I\otimes E)
\iso \catA(-,E).
\]
An argument analogous to the one above yields a natural isomorphism $-\otimes D^{\vee} \iso [D,-]$.
Dualizable objects are stable under direct sums, tensor products, internal Hom, and extensions.
That is, if $D$ and $E$ are dualizable objects,
then so are $D\oplus E$, $D\otimes E$, and $[D,E]$ (see \cite[Lemma 6.7]{HO}).
For an exact sequence $0 \to D \to X \to E \to 0$ such that $D$ and $E$ are dualizable,
the middle term $X$ is also dualizable (see \cite[Section 4, Lemma 7]{Bak}).

Let us explain another characterization of dualizable objects,
that is useful to prove several properties.
For objects $D$ and $E$,
the functor $-\otimes E$ is a right adjoint of $-\otimes D$
if and only if there are natural transformations $\id_{\catA} \to -\otimes D\otimes E$ and $-\otimes E\otimes D \to \id_\catA$ which satisfy the triangle identity.
It is also equivalent to the existence of morphisms $\eta\colon I \to D\otimes E$ and $\epsilon\colon E\otimes D \to I$ which make the following diagram commute:
\begin{equation}\label{diag:characterize dual}
\begin{tikzcd}
I \otimes D \ar[d,"\eta\otimes D"'] \ar[dr,"\iso"] & \\
D \otimes E \otimes D \ar[r,"D\otimes \epsilon"'] & D \otimes I,
\end{tikzcd}
\begin{tikzcd}
E \otimes I \ar[r,"E\otimes \eta"] \ar[rd,"\iso"'] & E \otimes D \otimes E \ar[d,"\epsilon \otimes E"]\\
& I \otimes E.
\end{tikzcd}
\end{equation}
Here, the diagonal arrows are the swap morphisms.
Composing the swap morphisms with $\eta$ and $\epsilon$, we can prove that $-\otimes E$ is also a left adjoint of $-\otimes D$.
Using this observation, we can prove that the canonical morphism $D \to D^{\vee\vee}$ is an isomorphism for any dualizable object $D$:
\[
{\catA}(-,D^{\vee\vee})
\iso {\catA}(- \otimes D^{\vee},I)
\iso {\catA}(- ,D).
\]
In particular, we have an isomorphism $-\otimes D \iso -\otimes D^{\vee\vee}\iso [D^{\vee},-]$,
which implies that $-\otimes D$ preserves colimits and limits if $D$ is dualizable.
When $\catA$ is an abelian category, it means that every dualizable object is flat,
that is, the functor $-\otimes D$ is exact for any dualizable object $D$.

Let $F\colon \catA \to \catB$ be a strong monoidal functor between symmetric monoidal closed categories.
Since $F$ preserves \eqref{diag:characterize dual},
for any dualizable object $D\in \catA$, the object $F(D)$ is also dualizable and $F(D^{\vee})\iso F(D)^{\vee}$ holds.

An object $L \in \catA$ is \emph{invertible} if there is $M\in \catA$ such that $L\otimes M \iso I$.
The set of isomorphism classes of invertible objects of $\catA$ is a group by the tensor product.
This group is called the \emph{Picard group} of $\catA$ and it is denoted by $\Pic \catA$.
If $L\otimes M \iso I$, then $-\otimes L$ is an autoequivalence on $\catA$ whose quasi-inverse is $-\otimes M$.
In particular, every invertible object is dualizable.

\subsection{Subcategories in symmetric monoidal abelian categories}
In this subsection,
let $\catA$ be a symmetric monoidal closed abelian category,
that is, $\catA$ is a symmetric monoidal closed category whose underlying category is an abelian category. 
Let us collect some basic definitions and properties of subcategories in symmetric monoidal closed abelian categories.
The material in this subsection seems to be well known to experts.
\begin{dfn}
Let $\catX$ be a subcategory of $\catA$.
\begin{enua}
\item
$\catX$ is a \emph{tensor ideal} if $X \otimes M \in \catX$ for any $X\in \catX$ and $M\in \catA$.
\item
$\catX$ is a \emph{Hom ideal} if $[M,X] \in \catX$ for any $X\in \catX$ and $M\in \catA$.
\item
$\catX$ is \emph{L-closed} if $X\otimes L \in \catX$ for any $X\in \catX$ and $L\in \Pic \catA$.
\end{enua}
\end{dfn}

If $\catA$ has a generating set, the above conditions can be checked using it as follows.
\begin{lem}\label{prp:criterion tensor CE}
Let $\{E_j\}_{j \in J}$ be a small generating set of $\catA$.
For a subcategory $\catX \subseteq \catA$ closed under coproducts and cokernels,
the following are equivalent:
\begin{enur}
\item
$\catX$ is a tensor ideal.
\item
For any $X \in \catX$ and $j\in J$,
we have $X\otimes E_j \in \catX$.
\end{enur}
\end{lem}
\begin{proof}
It is clear that (i) implies (ii).
We prove the converse.
For any $M \in \catA$,
take its presentation $\bigoplus_{j\in J} E_j^{\oplus A_j} \to \bigoplus_{j\in J} E_j^{\oplus B_j} \to M \to 0$ with respect to $\{E_j\}_{j \in J}$.
Tensoring any $X \in \catX$ with this presentation,
we obtain the right exact sequence $\bigoplus_{j\in J} (X\otimes E_j)^{\oplus A_j} \to \bigoplus_{j\in J} (X\otimes E_j)^{\oplus B_j} \to X \otimes M \to 0$.
As $\catX$ is closed under coproducts and cokernels,
we have $X \otimes M \in \catX$.
\end{proof}

\begin{lem}
Let $\{E_j\}_{j \in J}$ be a small generating set of $\catA$.
For a subcategory $\catX \subseteq \catA$ closed under products and kernels,
the following are equivalent:
\begin{enur}
\item
$\catX$ is a Hom ideal.
\item
For any $X \in \catX$ and $j\in J$,
we have $[E_j, X] \in \catX$.
\end{enur}
\end{lem}
\begin{proof}
This is proved in essentially the same way as \cref{prp:criterion tensor CE}.
\end{proof}

We call a localizing tensor ideal a \emph{tensor localizing subcategory}.
Similarly, we define \emph{tensor Serre subcategories}, \emph{tensor torsion classes}, and so on.
The set of tensor localizing subcategories is denoted by $\Loc_\otimes \catA$.
The set of tensor torsion classes is denoted by $\tors_{\otimes}\catA$.
The set of Hom ideal torsion-free classes is denoted by $\torf_{\Hom}\catA$.
A \emph{tensor torsion pair} is a torsion pair $(\catT,\catF)$ such that $\catT$ is a tensor ideal.
It can be characterized as follows.
\begin{prp}[{\cite[Proposition 2.4]{JLV}}]\label{prp:tensor tor pair}
The following are equivalent for a torsion pair $(\catT,\catF)$ in $\catA$:
\begin{enur}
\item
$[\catT,\catF]=0$ holds.
\item
$\catT$ is a tensor ideal.
\item
$\catF$ is a Hom ideal.
\end{enur}
\end{prp}
\begin{proof}
For the reader's convenience, we write down the proof.
The implication (i)$\imply$(ii) follows from
$\Hom_{\catA}(\catT\otimes \catA,\catF)=\Hom_{\catA}(\catA, [\catT,\catF])=0$.
The implication (ii)$\imply$(iii) follows from
$\Hom_{\catA}(\catT, [\catA,\catF])=\Hom_{\catA}(\catT\otimes \catA, \catF) \subseteq \Hom_{\catA}(\catT, \catF)=0$.
The implication (iii)$\imply$(i) follows from
$\Hom_{\catA}(\catA, [\catT,\catF])=\Hom_{\catA}(\catT, [\catA,\catF]) \subseteq \Hom_{\catA}(\catT, \catF)=0$.
\end{proof}

\subsection{Grothendieck categories with symmetric monoidal structures}
Following \cite{HO},
if the underlying category of a symmetric monoidal closed category is a Grothendieck category,
we call it a \emph{Grothendieck cosmos}.
In this subsection, we collect basic properties of Grothendieck cosmoi.
We often assume that a Grothendieck category is locally noetherian.
However, the arguments in this subsection also apply to the locally coherent case.

We first recall a criterion for a Grothendieck category equipped with a symmetric monoidal structure to be a Grothendieck cosmos.
\begin{fct}[{cf.\ \cite[A.2 Proposition]{BKS}}]\label{prp:criterion for Groth cosmos}
Let $\catA$ be a Grothendieck category equipped with a symmetric monoidal structure.
Then the symmetric monoidal category $\catA$ is closed, that is, 
the functor $-\otimes M$ admits a right adjoint for each $M\in \catA$,
if and only if $-\otimes M$ preserves small colimits for each $M\in \catA$.
\end{fct}

Following \cite{HO},
we consider the following conditions to ensure that the tensor product behaves well with respect to finitely presented objects:
\begin{cond}
Let $\catA$ be a Grothendieck cosmos.
\begin{description}
\item[(C1)]
The unit object $I$ is finitely presented.
\item[(C2)]
The Grothendieck category $\catA$ is generated by a small set of dualizable objects.
\end{description}
\end{cond}
The above conditions are satisfied by many explicit examples of Grothendieck cosmoi (see Section \ref{s:ex} and \cite[Section 6]{HO}).
\begin{fct}[{\cite[Proposition 6.9 and its proof]{HO}}]\label{fct:lfp base}
Let $\catA$ be a Grothendieck cosmos satisfying (C1) and (C2).
\begin{enua}
\item
If $M$ and $N$ are finitely presented, then so is $M\otimes N$.
\item
Every dualizable object is finitely presented.
\item
For any finitely presented object $M$,
there is an exact sequence $E_1 \to E_0 \to M \to 0$ such that $E_0$ and $E_1$ are dualizable.
\end{enua}
\end{fct}

A Grothendieck cosmos is said to be \emph{locally noetherian} 
if the underlying Grothendieck category is locally noetherian.
\begin{cor}\label{prp:noeth monoidal closed}
Let $\catA$ be a locally noetherian Grothendieck cosmos satisfying (C1) and (C2).
Then $\catA_{\noeth}$ is a symmetric monoidal closed abelian category.
\end{cor}
\begin{proof}
By (C1) and \cref{fct:lfp base} (1), the monoidal structure of $\catA$ restricts to $\catA_{\noeth}$.
It is enough to show that for any finitely presented objects $M$ and $N$,
the internal Hom $[M,N]$ is also finitely presented.
By \cref{fct:lfp base} (3), there is an exact sequence $E_1 \to E_0 \to M \to 0$
such that $E_0$ and $E_1$ are dualizable.
Then we have an exact sequence $0 \to [M,N] \to [E_0,N] \to [E_1,N]$.
Since $[E_i,N]\iso E_i^{\vee} \otimes N$ are finitely presented by \cref{fct:lfp base} (1) and (2),
their kernel $[M,N]$ is also finitely presented.
\end{proof}

\begin{rmk}\label{rmk:size of dualizable}
In the proof of \cref{fct:lfp base} (2) (\cite[Proposition 6.9]{HO}), we only use the condition (C1).
Thus, if $\catA$ is a locally noetherian Grothendieck cosmos satisfying (C1),
the subcategory of dualizable objects is essentially small since $\catA_{\noeth}$ is essentially small.
Thus, the following hold:
\begin{itemize}
\item
For a locally noetherian Grothendieck cosmos $\catA$ satisfying (C1),
it satisfies (C2) if and only if $\catA_{\noeth}$ has enough dualizable objects,
that is, for every object $X \in \catA_{\noeth}$, there exists a surjection $D\surj X$ from a dualizable object.
\item
The Picard group $\Pic \catA$ is isomorphic to a small group.
\end{itemize}
\end{rmk}

Next, let us review localization of locally noetherian Grothendieck cosmoi.
The following summarizes \cite[Appendix A]{BKS} in the locally noetherian case.
\begin{fct}[{\cite[Appendix A]{BKS}}]\label{fct:BKS}
Let $\catA$ be a locally noetherian Grothendieck cosmos and $\catX$ its tensor localizing subcategory.
Let $Q\colon \catA \to \catA/\catX$ denote the quotient functor,
and $R\colon \catA/\catX \to \catA$ its right adjoint.
Suppose that $\catA$ has enough flat objects.
\begin{enua}
\item
The Serre quotient $\catA/\catX$ is also a locally noetherian Grothendieck cosmos
such that the quotient functor $Q\colon \catA \to \catA/\catX$ is strict monoidal:
$Q(M\otimes_{\catA} N)=Q(M)\otimes_{\catA/\catX} Q(N)$.
\item
The quotient functor $Q$ preserves flat objects.
\item
There is a natural isomorphism of bifunctors:
\[
\catA[-,R-]\iso R((\catA/\catX)[Q-.-]).
\]
In particular, the right adjoint functor $R$ preserves the internal Hom.
\end{enua}
\end{fct}
Let $\catA$ be a locally noetherian Grothendieck cosmos satisfying (C1) and (C2),
and let $\catX$ be a tensor localizing subcategory of $\catA$.
Since $\catA$ has enough flat objects by (C2),
we can apply \cref{fct:BKS} to $\catA$.
Then $\catA/\catX$ is also a locally noetherian Grothendieck cosmos satisfying (C1) and (C2).

\section{Atom spectra with group actions}\label{s:Atom with G-act}
Throughout this section, let $\catA$ be an abelian category,
and suppose that a group $G$ acts on $\catA$,
that is, there is a group homomorphism from $G$ to the group of the isomorphism classes of autoequivalences of $\catA$.
The symbol $g\cdot X$ denotes the action of $g\in G$ on $X \in \catA$.
Note that the notion of action in this paper is different from that in \cite[Section 2.7]{EGNO}.

In this section, we introduce the \emph{$G$-orbit atom spectrum} $\Spec^G \catA$, the quotient space of the atom spectrum $\Spec \catA$ by $G$,
and investigate its basic properties and the localization of $\catA$ at a point of $\Spec^G \catA$ for later use.
When $G$ is trivial, the resulting theory agrees with the usual theory of the atom spectrum.
However, some of the results seem to be new even in this case.

Since any autoequivalence preserves monoform objects and atom-equivalence,
the group $G$ naturally acts on the atom spectrum $\Spec \catA$.
For $\alpha = \ol{H}\in \Spec \catA$ and $g \in G$,
we write $g\cdot \alpha := \ol{g\cdot H}$.
It is easy to see that $g \cdot \Supp(M)= \Supp(g\cdot M)$ and $g\cdot \Ass(M)= \Ass(g\cdot M)$
for any $M\in \catA$ and $g\in G$.
We denote by $\Spec^G \catA$ the quotient space of $\Spec \catA$ by $G$,
and call it the \emph{$G$-orbit atom spectrum} of $\catA$.
Let $\pi_G\colon \Spec \catA \to \Spec^G \catA$ be the quotient map.
Note that $\pi_G$ is an open surjective continuous map.
For $\alpha\in \Spec \catA$, we write $\wti{\alpha}:= \pi_G(\alpha) \in \Spec^G \catA$.
For a monoform object $H$, we also write $\wti{H}:=\pi_G(\ol{H})\in \Spec^G \catA$.
For an object $M$ of $\catA$, we write
\[
\Supp^G M := \pi_G(\Supp M),\quad
\Ass^G M := \pi_G(\Ass M).
\]
Then $\{\Supp^G M \mid M\in\catA \}$ forms an open basis of $\Spec^G\catA$.
We have $\Ass^G M \subseteq \Supp^G M$ by definition.
If $M$ is nonzero noetherian object, then $\Ass^G M$ is a non-empty finite set.
It is easy to see that $\Supp^G(g\cdot M)=\Supp^G(M)$ and $\Ass^G(g\cdot M)=\Ass^G(M)$ for any $M\in \catA$ and $g\in G$.
For a subcategory $\catX$ of $\catA$,
we put $\Supp^G \catX:=\bigcup_{X\in \catX}\Supp^G X$ and $\Ass^G \catX := \bigcup_{X\in \catX} \Ass^G X$.
%

Let $0\to X \to Y \to Z \to 0$ be an exact sequence,
and let $\{M_i\}_{i\in I}$ be a family of objects in $\catA$ such that the coproduct $\bigoplus_{i\in I} M_i$ exists.
The following formulas are a direct consequence of \cite[Proposition 3.3, 3.5, and 5.6]{Kan-Serre}:
\begin{equation}\label{eq:SuppG AssG formula}
\begin{aligned}
&\Supp^G Y =\Supp^G X \cup \Supp^G Z,\quad \Ass^G X \subseteq \Ass^G Y \subseteq \Ass^G X \cup \Ass^G Z,\\
&\Supp^G(\bigoplus_{i\in I} M_i)=\bigcup_{i\in I} \Supp^G M_i,\quad
\Ass^G(\bigoplus_{i\in I} M_i)=\bigcup_{i\in I} \Ass^G M_i
\end{aligned}
\end{equation}
The following lemma is useful to study $\Ass^G M$.
\begin{lem}\label{prp:ass decomp}
Let $M$ be a noetherian object.
Consider the decomposition $\Ass^G M = \Phi \sqcup \Psi$ as a set.
Then there is an exact sequence $0 \to L \to M \to N \to 0$
such that $\Ass^G L = \Phi$ and $\Ass^G N = \Psi$.
\end{lem}
\begin{proof}
Let $L$ be a maximal subobject of $M$ among those satisfying $\Ass^G L \subseteq \Phi$.
We prove that $\Ass^G L = \Phi$ and $\Ass^G(M/L)=\Psi$.
Take $x\in \Ass^G(M/L)$.
There is a monoform subobject $H$ of $M/L$ such that $\wti{H}=x$.
Pulling back the natural exact sequence $0 \to L \to M \to M/L \to 0$ via the inclusion morphism $H \inj M/L$, we obtain the following commutative diagram with exact rows:
\[
\begin{tikzcd}
0 \ar[r] & L \ar[d,equal] \ar[r] & L' \ar[r] \ar[d,hook] \pb & H \ar[r] \ar[d,hook] & 0\\
0 \ar[r] & L \ar[r] & M \ar[r] & M/L \ar[r] & 0.
\end{tikzcd}
\]
If $x \in \Phi$, then $\Ass^G L' \subseteq \Ass^GL \cup \{x\} \subseteq \Phi$,
which contradicts the maximality of $L$.
Thus, we have $\Ass^G(M/L) \subseteq \Psi$.
For the same reason, we obtain $\Ass^G L \subsetneq \Ass^G L'$.
Hence, we obtain $\Ass^G L' = \Ass^G L \cup \{x\}$.
This implies that $x\in \Ass^G L' \subseteq \Ass^G M$,
and we have $\Ass^G(M/L) \subseteq \Ass^G M$.
As $\Ass^G M \subseteq \Ass^G(L) \cup \Ass^G(M/L)$,
we can conclude that $\Ass^G L = \Phi$ and $\Ass^G(M/L)=\Psi$.
\end{proof}
%

A subcategory $\catX$ of $\catA$ is \emph{$G$-closed} if for any $g\in G$ and any $X\in \catX$,
we have $g\cdot X \in \catX$.
The set of $G$-closed localizing subcategories is denoted by $\Loc_{G}\catA$.
Similarly, we denote by $\tors_G\catA$ (resp.\ $\torf_G\catA$) the set of $G$-closed torsion classes (resp.\ $G$-closed torsion-free classes).
The subcategory $\catA_{\noeth}$ is $G$-closed,
and thus, $\catA_{\noeth}$ is also an abelian category with a $G$-action.
If a subset $\Phi$ of $\Spec \catA$ is $G$-stable, that is, $G \cdot \Phi \subseteq \Phi$, then $\Supp^{-1}(\Phi)$ and $\Ass^{-1}(\Phi)$ are $G$-closed.
For any subset $\Psi$ of $\Spec^G \catA$, define subcategories of $\catA$ by
\[
(\Supp^G)^{-1}(\Psi):=\{M\in \catA \mid \Supp^G M \subseteq \Psi\},\quad
(\Ass^G)^{-1}(\Psi) :=\{M\in \catA \mid \Ass^G M \subseteq \Psi\}.
\]
Then $(\Supp^G)^{-1}(\Psi)=\Supp^{-1}(\pi_G^{-1}(\Psi))$ and $(\Ass^G)^{-1}(\Psi)=\Ass^{-1}(\pi_G^{-1}(\Psi))$.
In particular, these are a $G$-closed Serre subcategory closed under coproducts and a $G$-closed torsion-free class closed under coproducts, respectively.

\begin{lem}\label{prp:classify G-closed localizing}
Let $\catA$ be a locally noetherian Grothendieck category with a $G$-action.
Then there exist bijections among the following sets:
\begin{enur}
\item 
the set of $G$-closed localizing subcategories of $\catA$.
\item 
the set of $G$-closed Serre subcategories of $\catA_{\noeth}$.
\item
the set of open subsets of $\Spec^G \catA$.
\end{enur}
The bijection between (i) and (iii) is given by $\catX \mapsto \Supp^G\catX$ and $\Phi \mapsto (\Supp^G)^{-1}(\Phi)$.
Here, $\catX$ is a $G$-closed localizing subcategory of $\catA$,
$\Supp^G\catX:=\bigcup_{X\in\catX} \Supp^G X$, and $\Phi$ is an open subset of $\Spec^G \catA$.
\end{lem}
\begin{proof}
By \cref{fct:Kanda classify Serre}, it is enough to show the following:
\begin{itemize}
\item
For any $G$-closed Serre subcategory $\catX$ of $\catA_{\noeth}$,
the smallest localizing subcategory $\catA$ containing $\catX$ is also $G$-closed.
\item
For any $G$-closed localizing subcategory $\catX$ of $\catA$,
the subcategory $\catX \cap \catA_{\noeth}$ is also $G$-closed.
\item
If $\catX$ is a $G$-closed localizing subcategory of $\catA$,
then $\Supp \catX$ is $G$-stable.
\item
If $\Phi$ is a $G$-stable open subset of $\Spec \catA$,
then $\Supp^{-1}(\Phi)$ is $G$-closed.
\item
The assignment $U\mapsto \pi_G^{-1}(U)$ yields a bijection between
the set of open subsets of $\Spec^{G}\catA$ and the set of $G$-stable open subsets of $\Spec \catA$.\end{itemize}
These are straightforward.
\end{proof}

A torsion pair $(\catT,\catF)$ is \emph{$G$-closed} if $\catT$ is $G$-closed.
Equivalently, $\catF$ is $G$-closed.
The following is a direct corollary of \cref{prp:hered tors by atom}.
\begin{cor}\label{prp:tensor torsion pair associated to subset}
Let $\catA$ be a locally noetherian Grothendieck category with a $G$-action.
For any open subset $\Phi$ of $\Spec^G \catA$,
the pair $((\Supp^G)^{-1}(\Phi), (\Ass^G)^{-1}(\Spec^G \catA \setminus \Phi))$ is a hereditary G-closed torsion pair. \qed
\end{cor}

If $\catX$ is a $G$-closed Serre subcategory of $\catA$,
then $G$ naturally acts on the Serre quotient $\catA/\catX$.
Moreover, if $\catX$ is a $G$-closed localizing subcategory of $\catA$,
there is a right adjoint $R\colon \catA/\catX \to \catA$ of the quotient functor $Q\colon \catA \to \catA /\catX$.
We see that $R(g\cdot N) \iso g \cdot R(N)$ for any $g\in G$ and $N \in \catA/\catX$ from the following calculation:
\[
\catA(-,R(g\cdot N))
\iso (\catA/\catX)(Q(-),g\cdot N)
\iso (\catA/\catX)(Q(g^{-1}\cdot -), N)
\iso \catA(-,g\cdot R(N)).
\]

\begin{lem}\label{prp:Spec^G A/X}
Let $\catA$ be a Grothendieck category with a $G$-action,
and let $\catX$ be a $G$-closed localizing subcategory of $\catA$.
The quotient functor $\catA \to \catA/\catX$ is denoted by $\pi$
and its right adjoint is denoted by $R$.
\begin{enua}
\item
$\Spec^G(\catA/\catX)$ is homeomorphic to $\Spec^G(\catA) \setminus \Supp^G(\catX)$.
\item
For every $M \in \catA$, we have
\[
\Ass^G_{(\catA/\catX)} Q(M) \supseteq \Ass^G_\catA M \setminus \Supp^G_{\catA} \catX,\quad
\Supp^G_{(\catA/\catX)} Q(M) = \Supp^G_{\catA} M \setminus \Supp^G_{\catA} \catX.
\]
\item
For every $N \in \catA/\catX$, we have
\[
\Ass^G_{(\catA/\catX)} N = \Ass^G_{\catA} R(N),\quad
\Supp^G_{(\catA/\catX)} N = \Supp^G_{\catA} R(N) \setminus \Supp^G_{\catA} \catX.
\]
\end{enua}
\end{lem}
\begin{proof}
(1):
The assignments $\ol{H} \mapsto \ol{Q(H)}$ and $\ol{H'} \mapsto \ol{R(H')}$ 
yield a homeomorphism between $\Spec (\catA/\catX)$ and $\Spec \catA \setminus \Supp \catX$ by \cref{fct:Spec of quot}.
As $Q(g\cdot M) = g \cdot Q(M)$ for any $M\in \catA$ and $g\in G$,
the assignment $\ol{H} \mapsto \ol{Q(H)}$ is $G$-equivariant.
Thus, taking the quotient of $\Spec (\catA/\catX)$ and $\Spec \catA \setminus \Supp \catX$,
we obtain a desired homeomorphism.

(2) and (3):
As $\catX$ is $G$-closed, we have $\Supp\catX=\pi^{-1}(\Supp^G\catX)$.
Using this observation, the assertion follows from \cref{fct:Spec of quot}.
\end{proof}

Next, we discuss the localization at a point of $\Spec^G \catA$.
In what follows, assume that $\catA$ is a locally noetherian Grothendieck category with a $G$-action.
For a point $x \in \Spec^G \catA$, put $U_{x}:=\Spec^G \catA \setminus \ol{\{x\}}$, where $\ol{\{x\}}$ is the closure of $\{x\}$ in $\Spec^G \catA$.
We write the corresponding G-closed localizing subcategory of $\catA$ by $\catX_{x}:=(\Supp^G)^{-1}(U_{x})$.
If $x=\wti{\alpha}$ for some $\alpha \in \Spec \catA$, we have
\[
\catX_{x}
=\Supp^{-1}(\pi_G^{-1}(U_{x}))
=\Supp^{-1}(\Spec\catA \setminus \ol{\pi_G^{-1}(x)})
=\Supp^{-1}(\Spec \catA \setminus \ol{G\cdot \alpha}).
\]
In the second equality, we use the fact that the quotient map $\pi_G \colon \Spec \catA \to \Spec^G \catA$ is an open surjective continuous map.
For $M \in \catA$, we note that
\begin{equation}\label{eq:catX_x}
M \in \catX_{x} 
\Equi \Supp^G M \subseteq U_{x}
\Equi \ol{\{x\}} \subseteq \Spec^G \catA \setminus \Supp^G M
\Equi x \not\in \Supp^G M.
\end{equation}
We define the \emph{localization of $\catA$ at $x$}
by the Serre quotient $\catA_{x}:=\catA/\catX_{x}$.
The quotient functor $\catA \to \catA_{x}$ is denoted by $(-)_{x}$.
Then $\catA_{x}$ is also a locally noetherian Grothendieck category with the natural $G$-action.
From \eqref{eq:catX_x}, we have
\[
\Supp^G M = \{x \in \Spec^G\catA \mid M_{x} \ne 0 \}.
\]
By \cref{prp:Spec^G A/X},
$\Spec^G(\catA_{x})$ is homeomorphic to $\ol{\{x\}}$.
This means that $\catA_{x}$ is $G$-local in the following sense.

\begin{dfn}
Let $\catA$ be a locally noetherian Grothendieck category with a $G$-action.
It is said to be \emph{$G$-local}
if $\Spec^G \catA$ has a generic point,
that is, there is $x \in \Spec^G\catA$ such that $\Spec^G\catA = \ol{\{x\}}$.
\end{dfn}
We now discuss some basic properties of $G$-local Grothendieck categories.
\begin{lem}\label{prp:G-local}
Let $\catA$ be a locally noetherian Grothendieck category with a $G$-action.
Suppose that $\catA$ is $G$-local, and let $x$ be a generic point of $\Spec^G\catA$.
\begin{enua}
\item
A generic point of $\Spec^G \catA$ is unique, and
we have $\wti{S}=x$ for any simple object $S$.
\item
For any $M \in \catA$ such that $x \in \Ass^G M$,
there is an injection $S \inj M$ from a simple object.
\end{enua}
\end{lem}
\begin{proof}
(1):
Let $S$ be a simple object.
Then $\Supp^G S$ contains the generic point $x$.
As the support of a simple object is a singleton,
so is $\Supp^G S$.
Thus, we have $\{\wti{S}\} = \Supp^G S = \{x\}$.
This proves $\wti{S}=x$.
Moreover, if $\Spec^G\catA = \ol{\{y\}}$ for some $y\in \Spec^G\catA$,
then $y \in \Supp^G S = \{x\}$.
Thus, a generic point is unique.

(2):
By the definition of $\Ass^G M$, there is a monoform subobject $H$ of $M$ such that $\wti{H}=x$.
Since $\catA$ is nonzero and locally noetherian, there is a simple object $S$.
As $\wti{H}=x=\wti{S}$ by (1), there is a nonzero common subobject of both $H$ and $g\cdot S$ for some $g\in G$.
This means that there is an injection $g\cdot S \inj H \inj M$ because $g\cdot S$ is simple.
\end{proof}

\begin{lem}\label{prp:char G-local}
For a nonzero locally noetherian Grothendieck category $\catA$ with a $G$-action,
the following are equivalent:
\begin{enur}
\item
$\catA$ is $G$-local.
\item
Any two simple objects are isomorphic up to $G$-action.
\end{enur}
\end{lem}
\begin{proof}
(i)$\imply$(ii):
Take two simple objects $S_1$ and $S_2$.
We have $\wti{S_1}=\wti{S_2}$ by \cref{prp:G-local} (1).
Then there exists some $g\in G$ such that $S_1$ and $g\cdot S_2$ have a common nonzero subobject.
Since $S_1$ and $S_2$ are simple, $S_1$ and $g\cdot S_2$ are isomorphic.

(ii)$\imply$(i):
Put $x:=\wti{S}$ for some simple object $S$.
Then $x=\wti{T}$ for all simple object $T$ by (ii).
Any nonzero object $M\in \catA$ has a simple subquotient since $\catA$ is locally noetherian.
Thus, we see that $x\in \Supp^G M$ for any nonzero $M\in \catA$.
As the subsets of the form $\Supp^G M$ are an open basis of $\Spec^G \catA$,
the point $x$ is a generic point of $\Spec^G \catA$.
\end{proof}

\begin{lem}\label{prp:length via Supp^G}
Let $\catA$ be a locally noetherian Grothendieck category with a $G$-action.
Suppose that $\catA$ is $G$-local and let $x$ be the generic point of $\Spec^G\catA$.
Then a nonzero noetherian object $M$ is of finite length if and only if $\Supp^G M = \{x\}$.
\end{lem}
\begin{proof}
The ``only if'' direction follows from \cref{prp:G-local} (1) and \cref{prp:char G-local} by using a composition series.
Suppose that $\Supp^G M = \{x\}$.
We have $\Ass^G M = \{x\}$, and there is an injection $S \inj M$ from a simple object by \cref{prp:G-local} (2).
If $M/S\ne0$, then $\Supp^G(M/S) = \{x\}$.
Therefore, iterating this procedure, we can conclude that $M$ is of finite length
since it is noetherian.
\end{proof}

It is known that the atom spectrum of an abelian category is a $T_0$-space,
that is, for any two distinct points, there is an open subset containing exactly one of them
(see \cite[Proposition 3.5]{Kan-spcl}).
We will show below that the orbit atom spectrum is also a $T_0$-space.

For a topological space $X$, define the \emph{specialization order} $\preceq$ on $X$ by
\[
x \preceq y \quad :\Longleftrightarrow \quad x\in \ol{\{y\}},
\]
where $\ol{\{y\}}$ denotes the topological closure of $\{y\}$.
Then $\preceq$ is always a preorder in general,
and it is a partial order if and only if $X$ is a $T_0$-space.
Let $Y$ be a subspace of $X$.
Then, for any subset $A$ of $Y$, we have $\ol{A}^X\cap Y = \ol{A}^{Y}$,
where $\ol{A}^X$ and $\ol{A}^Y$ are the closures of $A$ in $X$ and $Y$, respectively.
In particular, for any points $a, b\in Y$,
we have that $a \preceq b$ in $Y$ if and only if $a \preceq b$ in $X$.

\begin{prp}
Let $\catA$ be a locally noetherian Grothendieck category with a $G$-action.
Then $\Spec^G \catA$ is a $T_0$-space.
\end{prp}
\begin{proof}
Take any points $x,y \in \Spec^G\catA$ such that $\ol{\{x\}}=\ol{\{y\}}$.
As $\Spec^G \catA_x = \ol{\{x\}}=\ol{\{y\}}$, both $x$ and $y$ are generic points.
Thus, we obtain $x=y$ by \cref{prp:G-local} (1).
This proves $\Spec^G \catA$ is a $T_0$-space.
\end{proof}
Although the above proposition is easy to prove, it is nevertheless of interest.
Indeed, the quotient of a $T_0$-space by a group action need not be a $T_0$-space in general, and the action of $G$ on the abelian category is arbitrary.
This indicates that the fact that autoequivalences preserve simple objects imposes a strong constraint on the induced action on the atom spectrum.	

For later use,
we record the behavior of $\Supp^G$ and $\Ass^G$ under localization at a point.
\begin{cor}\label{prp:Supp^G Ass^G at x}
Let $\catA$ be a locally noetherian Grothendieck category with a $G$-action,
and let $x$ be a point of $\Spec^G \catA$.
We denote by $R\colon \catA_x \to \catA$ a right adjoint of the quotient functor $(-)_x\colon \catA \to \catA_x$.
\begin{enua}
\item
For every $M \in \catA$, we have
\[
\Ass^G_{\catA_x} M_x \supseteq \Ass^G_\catA M \cap \ol{\{x\}},\quad
\Supp^G_{\catA_x} M_x = \Supp^G_{\catA} M \cap \ol{\{x\}}.
\]
\item
For every $N \in \catA_x$, we have
\[
\Ass^G_{\catA_x} N = \Ass^G_{\catA} R(N),\quad
\Supp^G_{\catA_x} N = \Supp^G_{\catA} R(N) \cap \ol{\{x\}}.
\]
\end{enua}
\end{cor}
\begin{proof}
It follows from \cref{prp:Spec^G A/X}.
\end{proof}
We characterize when the equality $\Ass^G_{\catA_x} M_x = \Ass^G_\catA M \cap \ol{\{x\}}$ holds.
\begin{lem}\label{prp:equality for Ass M_x}
Let $\catA$ be a locally noetherian Grothendieck category with a $G$-action.
The following are equivalent:
\begin{enur}
\item
For any $M\in \catA$ and $x\in \Supp^G M$,
there exists $y \in \Ass^G M$ such that $y\in \ol{\{x\}}$.
\item
For any $G$-localizing subcategory $\catX$ and $M\in \catA$,
the equality $\Ass^G_{(\catA/\catX)} Q(M) = \Ass^G_\catA M \setminus \Supp^G_{\catA} \catX$ holds,
where $Q$ is the quotient functor.
\item
For any $x\in \Spec^G \catA$ and $M\in \catA$,
the equality $\Ass^G_{\catA_x} M_x = \Ass^G_\catA M \cap \ol{\{x\}}$ holds.
\end{enur}
These conditions are inherited by the quotient category $\catA/\catS$ for any $G$-localizing subcategory $\catS$.
\end{lem}
\begin{proof}
We can prove (i)$\imply$(ii) by an argument similar to \cite[Lemma 8.5]{Kan-CatSp}.
The implication (ii)$\imply$(iii) is obvious.
We prove (iii)$\imply$(i). Take $M\in \catA$ and $x\in \Supp^G M$.
As $M_x \ne 0$, there is $y \in \Ass^G_{\catA_x} M_x=\Ass^G_\catA M \cap \ol{\{x\}}$.
Thus, the condition (i) holds.
The final assertion follows since the condition (ii) is inherited by the quotient category $\catA/\catS$ for any $G$-localizing subcategory $\catS$.
\end{proof}

We will use the following lemma in \S \ref{ss:L-torf}.
\begin{lem}\label{prp:monoform sub with Tx simple}
Let $\catA$ be a locally noetherian Grothendieck category with a $G$-action.
Let $M \in \catA$ and $x\in \Ass^G M$.
Then there is a monoform subobject $T$ of $M$ such that
$\wti{T}=x$ and $T_x$ is a simple object in $\catA_x$.
\end{lem}
\begin{proof}
Take a monoform subobject $H$ of $M$ such that $\wti{H}=x$.
Since $\Ass_{\catA_x}^G(H_x)$ contains the generic point $x$ of $\Spec^G \catA_x$ by \cref{prp:Supp^G Ass^G at x}, there is an injection $i\colon S\inj H_x$ from a simple object in $\catA_x$.
Let $R$ be a right adjoint of the quotient functor $\catA \to \catA_x$.
Since $R$ is a right adjoint, the morphism $R(i)\colon R(S) \to R(H_x)$ is injective.
Pulling back $R(i)$ via the unit morphism $\eta\colon H \to R(H_x)$,
we obtain the following commutative diagram:
\[
\begin{tikzcd}
T \pb \ar[r,hook,"f"] \ar[d] & H \ar[d,"\eta"] \\
R(S) \ar[r,hook,"R(i)"'] & R(H_x).
\end{tikzcd}
\]
As the quotient functor is exact, it preserves the pullback diagram.
Thus, we have $T_x \iso R(S)_x \iso S$.
Moreover, since $T$ is a nonzero subobject of $H$, it is a monoform object atom-equivalent to $H$.
\end{proof}
The monoform object $T$ in the above lemma plays the role of the $R$-module $R/\pp$,
where $\pp$ is a prime ideal of a commutative ring $R$.

\section{Atom spectra of symmetric monoidal abelian categories}\label{s:LSpec}
In this section,
we apply the theory of orbit atom spectra developed in \S \ref{s:Atom with G-act} to symmetric monoidal abelian categories,
and show that, under suitable assumptions, monoidal abelian localizations can be controlled by orbit atom spectra.
This section provides the technical foundation for the next section, where we classify various subcategories of symmetric monoidal abelian categories.

Let $\catA$ be a symmetric monoidal closed abelian category.
Then the Picard group $\Pic \catA$ acts on $\catA$ by tensor product.
Thus, as in the introduction (\S \ref{s:intro}), we can apply the results of \S \ref{s:Atom with G-act} with $G=\Pic \catA$, and we write
\[
\LSpec \catA := \Spec^{\Pic \catA} \catA,\quad
\LSupp(M) := \Supp^{\Pic \catA} (M),\quad
\LAss(M) := \Ass^{\Pic \catA} (M).
\]
The choice $G=\Pic \catA$ is indeed natural.
However, in concrete examples, it is often more convenient to take $G$ to be a subgroup of $\Pic \catA$.
For this reason, throughout this section, we work in the more general setting in which a group $G$ acts on a symmetric monoidal abelian category $\catA$.

First, we introduce a basic setup that satisfies the assumptions of this section and the next section.
This setup is a more refined version of \cref{setup:A} from the introduction (\S \ref{s:intro}).
\begin{setup}\label{setup:classify subcat}
Let $\catA$ be a symmetric monoidal closed abelian category satisfying the following conditions:
\begin{itemize}
\item
The unit object $I$ is finitely presented.
\item
There exists a small set $\calL$ of invertible objects such that
$\catA$ is generated by a small set of $\calL^{\otimes}$-filtered objects,
where $\calL^{\otimes}:=\{L^{\otimes n} \mid L \in \calL, n\in \bbZ\}$.
\end{itemize}
Let $G_{\calL}$ be the subgroup of $\Pic \catA$ generated by the isomorphism classes of objects of $\calL$.
Then $G_{\calL}$ acts on $\catA$ via tensor product.
\end{setup}
See Section \ref{s:ex} for concrete examples of Grothendieck cosmoi satisfying the conditions in \cref{setup:classify subcat}.
Note that the conditions in \cref{setup:classify subcat} are inherited by the quotient category $\catA/\catX$ for any tensor localizing subcategory $\catX$.
The conditions (C1) and (C2) are clearly satisfied under \cref{setup:classify subcat}.

Under \cref{setup:classify subcat}, for various classes of subcategories, being a tensor ideal or a Hom ideal is equivalent to being $L$-closed.
In particular, these conditions are controlled by the autoequivalences induced by invertible objects.
\begin{lem}\label{prp:criterion for D-closed}
Let $\catA$ be a symmetric monoidal closed abelian category satisfying \cref{setup:classify subcat}, and let $\calL$ be the set of invertible objects appearing in \cref{setup:classify subcat}.
Set $G:=G_{\calL}$.
\begin{enua}
\item
For any CE-closed subcategory $\catX$ closed under coproducts, the following are equivalent:
\begin{enur}
\item $\catX$ is a tensor ideal.
\item $\catX$ is $L$-closed.
\item $\catX$ is $G$-closed.
\item $X\otimes L \in \catX$ for any $X\in \catX$ and any $L\in\calL$.
\end{enur}
\item
For any KE-closed subcategory $\catX$ closed under products, the following are equivalent:
\begin{enur}
\item $\catX$ is a Hom ideal.
\item $\catX$ is $L$-closed.
\item $\catX$ is $G$-closed.
\item $X\otimes L^{\vee} \in \catX$ for any $X\in \catX$ and any $L\in\calL$.
\end{enur}
\end{enua}
\end{lem}
\begin{proof}
(1):
Let $\catX$ be a CE-closed subcategory closed under coproducts.
It is obvious that the implications (i)$\imply$(ii)$\imply$(iii)$\imply$(iv) hold.
We prove (iv)$\imply$(i).
Let $\{E_k\}_k$ be a small generating set of $\catA$ such that $E_k$ is $\calL^{\otimes}$-filtered for every $k$.
By \cref{prp:criterion tensor CE},
the subcategory $\catX$ is a tensor ideal if and only if $X\otimes E_k \in \catX$ for any $X\in \catX$ and $k$.
There are invertible objects $L_{i_1}, \dots, L_{i_m}$ such that $E_k\in L_{i_1}* \dots * L_{i_m}$ for a fixed $k$.
Then we have $X \otimes E_k \in (X\otimes L_{i_1})* \dots * (X\otimes L_{i_m}) \subseteq \catX$
by \cref{prp:ex seq with flat right term} below.
Thus, $\catX$ is a tensor ideal.

(2):
As dualizable objects are closed under extensions, $\calL^{\otimes}$-filtered objects are also dualizable.
For any dualizable object $D$ and $X\in \catX$, we have $[D,X] \iso D^{\vee}\otimes X$.
With these observations in mind, the proof is the same as that of (1).
\end{proof}
\begin{lem}\label{prp:ex seq with flat right term}
Let $\catA$ be a symmetric monoidal abelian category.
Suppose that $\catA$ has enough flat objects,
that is, for any $X\in \catA$, there exists a surjection $F \surj X$ from a flat object.
Then for any exact sequence $0 \to M \to N \to E \to 0$ such that $E$ is flat,
the sequence $0 \to M \otimes Z \to N\otimes Z \to E\otimes Z \to 0$ remains exact
for any $Z \in \catA$.
\end{lem}
\begin{proof}
The proof is essentially the same as that of \cite[\href{https://stacks.math.columbia.edu/tag/00HL}{Tag 00HL}]{SP}.
\end{proof}
In particular, for a locally noetherian Grothendieck cosmos satisfying \cref{setup:classify subcat},
the equality $\Loc_{\otimes}\catA=\Loc_G \catA$ holds,
that is, a localizing subcategory of $\catA$ is a tensor ideal if and only if $G$-closed.
From \cref{prp:classify G-closed localizing},
this means that tensor localizing subcategories are classified by open subsets of $\Spec^G \catA$,
and thus, tensor abelian localizations are controlled by $\Spec^G \catA$.
This observation is one of the motivations for introducing the orbit atom spectrum $\Spec^G \catA$.

In what follows, we consider a locally noetherian Grothendieck cosmos $\catA$ with a $G$-action satisfying (C1), (C2), and $\Loc_{\otimes}\catA=\Loc_G \catA$.
%
Take a point $x \in \Spec^G \catA$.
Since every G-closed localizing subcategory is a tensor localizing subcategory,
the subcategory $\catX_{x}$ introduced in Section \ref{s:Atom with G-act} is a tensor localizing subcategory.
Thus, by \cref{fct:BKS}, the localization $\catA_{x}$ at $x$ is also a locally noetherian Grothendieck cosmos with a $G$-action satisfying (C1) and (C2).
Note that $(M\otimes_{\catA} N)_x = M_x \otimes_{\catA_x} N_x$ holds for any $M, N \in \catA$.
Let us compare the internal Hom of $\catA$ and $\catA_{x}$.

\begin{lem}
Let $\catA$ be a Grothendieck cosmos satisfying (C1) and (C2),
and let $F\colon \catA \to \catB$ be an exact and strong monoidal functor to another symmetric monoidal closed abelian category.
Then for any finitely presented object $M \in \catA$ and any object $N\in \catA$, we have an isomorphism $F(\catA[M,N])\iso \catB[FM,FN]$ in $\catB$.
\end{lem}
\begin{proof}
There is an exact sequence $E_1 \to E_0 \to M \to 0$ such that $E_i$ is dualizable
by \cref{fct:lfp base} (3).
Note that we have the following natural isomorphisms for any dualizable object $E$:
\[
F(\catA[E,N]) \iso F(E^{\vee} \otimes N) \iso F(E)^{\vee} \otimes F(N) \iso \catB[FE,FN].
\]
Thus, we have the following commutative diagram with exact rows:
\[
\begin{tikzcd}
0 \ar[r] & F(\catA[M,N]) \ar[r] & F(\catA[E_0,N]) \ar[r] \ar[d,"\iso"] & F(\catA[E_1,N]) \ar[d,"\iso"]\\
0 \ar[r] & \catB[FM,FN]\ar[r] & \catB[FE_0,FN] \ar[r] & \catB[FE_1,FN].
\end{tikzcd}
\]
By the uniqueness of kernels, we obtain the desired isomorphism.
\end{proof}

\begin{cor}
Let $\catA$ be a locally noetherian Grothendieck cosmos with a $G$-action satisfying (C1), (C2), and $\Loc_{\otimes}\catA=\Loc_G \catA$.
Take a point $x\in \Spec^G \catA$.
For any finitely presented object $M\in \catA$ and any object $N \in \catA$, we have an isomorphism
$\catA[M,N]_{x}=\catA_{x}[M_{x},N_{x}]$ in $\catA_{x}$.
\end{cor}

Next, let us study the compatibility of supports and tensor construction.
\begin{lem}
Let $\catA$ be a locally noetherian Grothendieck cosmos with a $G$-action satisfying (C1), (C2), and $\Loc_{\otimes}\catA=\Loc_G \catA$.
\begin{enua}
\item
$\Supp^G(M\otimes N) \subseteq \Supp^G M \cap \Supp^G N$ holds.
\item
If $M$ is finitely presented,
then $\Supp^G [M,N] \subseteq \Supp^G M \cap \Supp^G N$ holds.
\end{enua}
\end{lem}
\begin{proof}
It follows from $(M\otimes_{\catA}N)_x \iso M_x\otimes_{\catA_x}N_x$
and $\catA[M,N]_{x}\iso \catA_{x}[M_{x},N_{x}]$.
\end{proof}

Let us study when the equality $\Supp^G(M\otimes N)=\Supp^G M \cap \Supp^G N$ holds.
The following result will not be used in the rest of this paper.
\begin{prp}
Let $\catA$ be a locally noetherian Grothendieck cosmos satisfying (C1), (C2) and $\Loc_{\otimes}\catA=\Loc_G \catA$.
Suppose that $g\cdot(M\otimes N)\iso (g\cdot M)\otimes N \iso M\otimes (g\cdot N)$ for any $g\in G$ and $M,N \in \catA$.
Then $\Supp^G(M\otimes N)=\Supp^G M \cap \Supp^G N$ holds for any $M, N \in \catA_{\noeth}$
if and only if $S\otimes_{\catA_x} S \ne 0$ for any $x\in \Spec^G \catA$ and any simple object $S\in \catA_x$.
\end{prp}
\begin{proof}
Suppose that $\Supp^G(M\otimes N)=\Supp^G M \cap \Supp^G N$ holds for any $M, N \in \catA_{\noeth}$.
Take a point $x\in \Spec^G \catA$ and a simple object $S$ of $\catA_x$.
Since $S \in (\catA_x)_{\noeth}=\catA_{\noeth}/(\catX_x\cap \catA_{\noeth})$,
there is $T \in \catA_{\noeth}$ such that $T_x \iso S$.
Then we have $x \in \Supp^G T = \Supp^G(T\otimes_\catA T)$,
and hence $S\otimes_{\catA_x} S = (T\otimes_\catA T)_x \ne 0$.

Conversely, suppose that $S\otimes_{\catA_x} S \ne 0$ for any $x\in \Spec^G \catA$ and any simple object $S\in \catA_x$.
Take $M, N \in \catA_{\noeth}$ and $x\in \Spec^G \catA$.
It is enough to show that if $M_x \ne 0$ and $N_x \ne 0$, then $(M \otimes N)_x \ne 0$.
As $M_x,N_x \in (\catA_{x})_{\noeth}$,
there are surjections $M_x \surj S$ and $N_x \surj S'$ to simple objects.
Then we obtain a surjection $(M \otimes N)_x \iso M_x \otimes N_x \surj S \otimes S'$.
Since $\catA_x$ is $G$-local, there is $g \in GA$ such that $S' \iso g\cdot S$.
This means that $S\otimes S' \iso g\cdot(S\otimes S)$ is nonzero,
and hence $(M \otimes N)_x\ne 0$.
\end{proof}
For example, for a commutative local ring $(R,\mm)$, we have $(R/\mm) \otimes_R (R/\mm)=R/\mm$.
This is a reason why we have the equality $\Supp(M\otimes_R N)=\Supp M \cap \Supp N$ for modules $M$ and $N$ over a commutative ring.

\section{Classifying subcategories of symmetric monoidal abelian categories}\label{s:classify subcat}\label{s:classify subcat}
In this section,
we classify several classes of subcategories of symmetric monoidal closed noetherian abelian categories
by applying the theory developed in \S \ref{s:Atom with G-act} and \ref{s:LSpec}.
Throughout this section,
let $\catA$ be a locally noetherian Grothendieck cosmos with a $G$-action.

To make the assumptions used explicit in this section, we introduce several conditions.
All of them, except for ($*$) in \cref{prp:Ass^G X controlls X},
are satisfied under \cref{setup:classify subcat}.
\begin{lem}
Let $\catA$ be a locally noetherian Grothendieck cosmos satisfying \cref{setup:classify subcat}.
\begin{itemize}
\item 
$\Loc_{\otimes} \catA = \Loc_G \catA$ holds.
\item
$\tors_{\otimes} \catA_{\noeth} = \tors_G \catA_{\noeth}$ holds.
\item
$\torf_{\Hom} \catA_{\noeth} = \torf_G \catA_{\noeth}$ holds.
\end{itemize}
\end{lem}
\begin{proof}
Since $\catA$ satisfies (C1) and (C2), the noetherian abelian category $\catA_{\noeth}$ is also a symmetric monoidal closed category by \cref{prp:noeth monoidal closed}.
Then $\catA_{\noeth}$ is also generated by a small set of $\calL^{\otimes}$-filtered objects
as $\calL^{\otimes}$-filtered objects are finitely presented.
Thus, the assertion follows from \cref{prp:criterion for D-closed}.
\end{proof}

\subsection{Every tensor torsion class is a Serre subcategory}\label{ss:L-tors=L-Serre}
In this subsection,
we prove the title of this subsection.
\begin{prp}\label{prp:Serre=tors}
Suppose that $\catA$ satisfies (C1), (C2), $\Loc_{\otimes} \catA = \Loc_G \catA$, and $\tors_{\otimes} \catA_{\noeth} = \tors_G \catA_{\noeth}$.
For any $G$-closed subcategory $\catT$ in $\catA_{\noeth}$, the following are equivalent:
\begin{enur}
\item
$\catT$ is a torsion class.
\item
$\catT$ is a Serre subcategory.
\end{enur}
\end{prp}
\begin{proof}
We only prove (i)$\imply$(ii) as the converse is clear. 
Let $\catT$ be a $G$-closed torsion class and put $\Phi:=\Supp^G\catT$.
Put $\catX_{\Phi}:=(\Supp^G)^{-1}_{\catA_{\noeth}}(\Phi)$
and $\catY_{\Phi}:=(\Ass^G)^{-1}_{\catA_{\noeth}}(\Spec^G \catA_{\noeth}\setminus \Phi)$.
Then $(\catX_{\Phi},\catY_{\Phi})$ is a $G$-closed hereditary torsion pair of $\catA_{\noeth}$
by \cref{prp:tensor torsion pair associated to subset}.
We want to show that $\catT = \catX_{\Phi}$.
It is clear that $\catT \subseteq \catX_{\Phi}$.
Put $\catF:={\catT}^{\perp}$.
It is enough to show that $\catF \subseteq \catY_{\Phi}$,
that is, $\Supp^G\catT \cap \Ass^G \catF$ is empty.
Take $x\in \Ass^G \catF$.
There is a monoform object $H \in \catF$ such that $\wti{H}=x$ and $H_x$ is simple in $\catA_x$ by \cref{prp:monoform sub with Tx simple}.
For any $T\in \catT$, we have the following as $\catA[\catT,\catF]=0$ by \cref{prp:tensor tor pair}:
\[
\catA_{x}[T_{x},H_{x}]\iso \catA[T,H]_{x}=0.
\]
Replacing $H$ with $g\cdot H$ for any $g\in G$,
we can conclude that $\catA_x(T_x,S)=0$
for any simple object $S$ of $\catA_x$ since $\catA_x$ is $G$-local
(see \cref{prp:char G-local}).
This implies $T_x=0$ for any $T\in\catT$ by \cref{prp:Nakayama},
and hence $x \not\in \Supp^G\catT$.
%
\end{proof}

\begin{cor}\label{prp:Serre=CE}
Suppose that $\catA$ satisfies (C1), (C2), $\Loc_{\otimes} \catA = \Loc_G \catA$, and $\tors_{\otimes} \catA_{\noeth} = \tors_G \catA_{\noeth}$.
Let $\catX$ be a $G$-closed subcategory in $\catA_{\noeth}$.
\begin{enua}
\item
The following are equivalent in $\catA_{\noeth}$:
\begin{itemize}
\item $\catX$ is a Serre subcategory.
\item $\catX$ is a torsion class.
\item $\catX$ is a wide subcategory.
\item $\catX$ is an ICE-closed subcategory.
\item $\catX$ is a CE-closed subcategory.
\end{itemize}
\item
The following are equivalent in $\catA_{\noeth}$:
\begin{itemize}
\item $\catX$ is a torsion-free class.
\item $\catX$ is an IKE-closed subcategory.
\item $\catX$ is an IE-closed subcategory.
\end{itemize}
\end{enua}
\end{cor}
\begin{proof}
The smallest torsion(-free) class containing $\catX$ is called the torsion(-free) closure of $\catX$.
We can see that the torsion(-free) closure of $\catX$ is also $G$-closed by the construction
(cf.\ \cite[Proposition 2.3 and Lemma 3.22]{Sai-coh}).
We can prove the assertion by an argument similar to \cite[Proposition 5.1]{KS}.
\end{proof}

\subsection{Classifying $G$-closed torsion-free classes}\label{ss:L-torf}
In this subsection, we classify the $G$-closed torsion-free classes of $\catA_{\noeth}$ in terms of the subsets of $\Spec^G\catA$.

To prove the following lemma, we recall the construction of natural morphisms in a symmetric monoidal closed category $\catB$.
Let $X$ and $Y$ be objects of $\catB$.
Since the functor $[X,-]$ is a right adjoint of $-\otimes X$,
we obtain the \emph{evaluation morphism} $[X,Y]\otimes X \to Y$ corresponding to $\id_{[X,Y]}$ under the adjunction.
Applying the adjunction once again to the composite $X\otimes [X,Y] \iso [X,Y]\otimes X \to Y$,
we obtain a natural map $X \to [[X,Y],Y]$.
Moreover, if the underlying category of $\catB$ is preadditive,
the natural map $X \to [[X,Y],Y]$ is nonzero if and only if $[X,Y]$ is nonzero.
\begin{lem}\label{prp:Ass^G X controlls X}
Suppose that $\catA$ satisfies (C1), (C2), $\Loc_{\otimes}\catA = \Loc_G \catA$, and $\torf_{\Hom} \catA_{\noeth} = \torf_G \catA_{\noeth}$.
In addition, suppose that the following condition holds (cf.\ \cref{prp:equality for Ass M_x}):
\begin{description}
\item[($*$)]
For any $M\in \catA$ and $x\in \Supp^G M$, there exists $y\in \Ass^G M$ such that $y \in \ol{\{x\}}$.
\end{description} 
Let $\catX$ be a $G$-closed torsion-free class of $\catA_{\noeth}$ and $x\in \Ass^G \catX$.
For any object $M \in \catA_{\noeth}$,
if $\Ass^G(M)=\{x\}$, then $M\in \catX$.
\end{lem}
\begin{proof}
Let $M$ be an object such that $\Ass^G(M)=\{x\}$.
By the condition ($*$) and \cref{prp:equality for Ass M_x}, we have $\Ass_{\catA_x}^G M_x = \Ass^G M\cap \ol{\{x\}}=\{x\}$.
Thus, $M_x$ is of finite length by \cref{prp:length via Supp^G}.
We prove $M\in \catX$ by induction on the length $\ell_{\catA_x}(M_x)$ of $M_x$.
Take a simple quotient $M_x \surj S$.
As $x\in \Ass^G(\catX)$ and $\catX$ is $G$-closed,
there exists a monoform object $T \in \catX$ such that $T_x\iso S$ by \cref{prp:char G-local,prp:monoform sub with Tx simple}.
Consider the natural morphism $\phi\colon M \to [[M,T],T]$.
Then $[[M,T],T] \in \catX$ as $\catX$ is a Hom ideal, and $\phi_x\colon M_x \to [[M_x,S],S]$ is nonzero since there is the surjection $M_x \surj S$.
Put $K:=\Ker(\phi)$.
Then $\ell_{\catA_x}(K_x) < \ell_{\catA_x}(M_x)$,
and $\Ass^G K \subseteq \Ass^G M=\{x\}$.
If $\ell_{\catA_x}(M_x)=1$, then $K_x=0$, and hence $K=0$.
This means that $\phi$ is injective, and $M \in \catX$.
If $\ell_{\catA_x}(M_x)\ge2$, then $K\in \catX$ by the induction hypothesis.
Then $M \in \catX$ by the exact sequence $0 \to K \to M \to [[M,T],T]$.
\end{proof}

\begin{thm}\label{prp:classify G-torf}
Suppose that $\catA$ satisfies (C1), (C2), $\Loc_{\otimes}\catA = \Loc_G \catA$, and $\torf_{\Hom} \catA_{\noeth} = \torf_G \catA_{\noeth}$.
In addition, suppose that $\catA$ satisfies ($*$) in \cref{prp:Ass^G X controlls X}.
Then there is an order-preserving bijection between the following sets:
\begin{enur}
\item
the set of $G$-closed torsion-free classes of $\catA_{\noeth}$.
\item
the power set of $\Spec^G \catA$.
\end{enur}
The bijection is given by $\Phi \mapsto (\Ass^{G})^{-1}(\Phi)$ and $\catX \mapsto \Ass^G\catX$,
where $\Phi$ is a subset of $\Spec^G \catA$ and $\catX$ is a $G$-closed torsion-free class.
\end{thm}
\begin{proof}
We first prove that $\Ass^G((\Ass^G)^{-1}(\Phi))=\Phi$ for any subset $\Phi$ of $\Spec^G\catA$.
The inclusion $\Ass^G((\Ass^G)^{-1}(\Phi))\subseteq \Phi$ is clear.
For any monoform object $H$, we have $\Ass^G(H)=\{\wti{H}\}$ by definition.
Thus, the converse inclusion holds.

Next, we prove that $(\Ass^G)^{-1}(\Ass^G \catX)=\catX$ for any $G$-closed torsion-free class $\catX$ of $\catA_{\noeth}$.
It is clear that $(\Ass^G)^{-1}(\Ass^G \catX) \supseteq \catX$ holds.
Let $M$ be an object such that $\Ass^G M \subseteq \Ass^G\catX$.
To prove $M \in \catX$, we may assume that $\Ass^G M$ is a singleton by \cref{prp:ass decomp}.
Then $M\in \catX$ follows from \cref{prp:Ass^G X controlls X}.
\end{proof}

\section{Examples}\label{s:ex}
In this section, we describe the orbit atom spectrum for several concrete Grothendieck cosmoi satisfying \cref{setup:classify subcat} and discuss consequences of \cref{prp:Serre=CE,prp:classify G-torf}.
In \S \ref{ss:non ex}, we give examples of Grothendieck cosmoi that lie outside the scope of \cref{setup:classify subcat} but nevertheless have the property that tensor torsion classes are Serre subcategories.
\subsection{The category of modules}\label{ss:Mod}
Let $A$ be a ring.
The category $\Mod A$ of (right) $A$-modules is a Grothendieck category with a generator $A$.
The Grothendieck category $\Mod A$ is locally noetherian if and only if $A$ is right noetherian.
Then $(\Mod A)_{\noeth}$ is nothing but the category $\catmod A$ of finitely generated $A$-modules.

If $A$ is commutative, then $\Mod A$ becomes a symmetric monoidal closed category with the usual tensor product over $A$, and hence, $\Mod A$ is a Grothendieck cosmos.
The dualizable objects in $\Mod A$ are precisely the finitely generated projective modules.
Indeed, if $D$ is a dualizable object of $\Mod A$,
then the functor $\Hom_A(D,-)\iso -\otimes D^{\vee}$ preserves surjectivity and directed colimits.
This means that $D$ is projective and finitely presented.
The converse direction is clear.
Thus, the Grothendieck cosmos $\Mod A$ always satisfies (C1) and (C2).

Let $R$ be a commutative noetherian ring.
Then $\Mod R$ satisfies the assumptions in \cref{setup:classify subcat},
and $G_\calL$ in \cref{setup:classify subcat} is the trivial group, where $\calL=\{R\}$.
By \cite[Proposition 7.2]{Kan-Serre}, there is a bijection 
\[
\Spec R \isoto \Spec(\Mod R),\quad \pp \mapsto \ol{R/\pp},
\]
where $\Spec R$ is the set of prime ideals of $R$.
Thus, $\Mod R$ satisfies the condition ($*$) of \cref{prp:Ass^G X controlls X}.
Therefore, we can recover the following known results from \cref{prp:Serre=CE,prp:classify G-torf}.
\begin{cor}[{\cite{Takahashi,SW,IMST,Eno}}]\label{prp:tors=Serre in mod R}
Let $R$ be a commutative noetherian ring,
and let $\catX$ be a subcategory of $\catmod R$.
\begin{enua}
\item
The following are equivalent:
\begin{itemize}
\item $\catX$ is a Serre subcategory.
\item $\catX$ is a torsion class.
\item $\catX$ is a wide subcategory.
\item $\catX$ is an ICE-closed subcategory.
\item $\catX$ is a CE-closed subcategory.
\end{itemize}
\item
The following are equivalent:
\begin{itemize}
\item $\catX$ is a torsion-free class.
\item $\catX$ is an IKE-closed subcategory.
\item $\catX$ is an IE-closed subcategory.
\end{itemize}
\end{enua}
\end{cor}

\begin{cor}[{\cite{Takahashi}}]
Let $R$ be a commutative noetherian ring.
There is an order-preserving bijection between the following sets:
\begin{itemize}
\item
The set of torsion-free classes of $\catmod R$.
\item
The power set of $\Spec R$.
\end{itemize}
\end{cor}

\subsection{The category of graded modules}
For the basics of graded ring theory,
we refer to \cite{NO}.
Throughout this section, fix an abelian group $G$.

Let $A$ be a $G$-graded ring.
The category $\Mod^G A$ of graded (right) $A$-modules is a Grothendieck category with a generating set $\{A(g)\}_{g\in G}$,
where $(g)$ is the grading shift functor $\Mod^G A \isoto \Mod^G A$.
The Grothendieck category $\Mod^G A$ is locally noetherian
if and only if $A$ is \emph{$G$-noetherian}, that is, every ascending chain of homogeneous ideals is stationary.
Then $(\Mod^G A)_{\noeth}$ is nothing but the category $\catmod^G A$ of finitely generated graded $A$-modules.

If $A$ is commutative, that is, its underlying ring is commutative,
then $\Mod^G A$ becomes a symmetric monoidal closed category with the usual graded tensor product over $A$, and hence, $\Mod^G A$ is a Grothendieck cosmos.
The dualizable objects in $\Mod^G A$ are precisely the finitely generated projective graded modules, that is, the direct summands of finite direct sums of $\{A(g)\}_{g\in G}$, by an argument similar to that in \S \ref{ss:Mod}.
Thus, the Grothendieck cosmos $\Mod^G A$ always satisfies the conditions (C1) and (C2).

Let $R$ be a commutative $G$-noetherian $G$-graded ring
and put $\calL=\{R(g)\}_{g\in G}$.
Then $\Mod^G R$ satisfies the assumptions in \cref{setup:classify subcat}.
The grading group $G$ acts on $\Mod^G R$ by grading shifts.
This action coincides with an action given 
by the natural surjection $G \surj G_{\calL}$ to the group defined in \cref{setup:classify subcat}.
A $G$-closed subcategory is precisely a subcategory closed under grading shifts,
that is, if $X$ belongs to the subcategory, then so does $X(g)$ for any $g\in G$.
By \cite[Corollary 3.4]{Posur}, there is a surjection
\[
\Spec^G R \times G \surj \Spec(\Mod^G R),\quad (\pp,g) \mapsto \ol{(R/\pp)(g)},
\]
where $\Spec^G R$ is the set of homogeneous prime ideals of $R$.
From this, we obtain the following bijection by \cite[Lemma 3.6]{Posur}:
\[
\Spec^G R \isoto \Spec^G(\Mod^G R),\quad \pp \mapsto \wti{R/\pp}.
\]
Note that the \emph{atom projection} $\Spec(\Mod^G R) \to \Spec^G R$ in \cite{Posur} is nothing but the natural quotient map $\Spec(\Mod^G R) \to \Spec^G(\Mod^G R)$.
From \cite[Theorem and Definition 3.8 (2)]{Posur},
we can easily see that $\Mod^G R$ satisfies the condition ($*$) of \cref{prp:Ass^G X controlls X}.
Therefore, we obtain the following results, which seem to be well known to experts.
\begin{cor}
Let $R$ be a commutative $G$-noetherian $G$-graded ring,
and let $\catX$ be a subcategory of $\catmod^G R$ closed under grading shifts.
\begin{enua}
\item
The following are equivalent:
\begin{itemize}
\item $\catX$ is a Serre subcategory.
\item $\catX$ is a torsion class.
\item $\catX$ is a wide subcategory.
\item $\catX$ is an ICE-closed subcategory.
\item $\catX$ is a CE-closed subcategory.
\end{itemize}
\item
The following are equivalent:
\begin{itemize}
\item $\catX$ is a torsion-free class.
\item $\catX$ is an IKE-closed subcategory.
\item $\catX$ is an IE-closed subcategory.
\end{itemize}
\end{enua}
\end{cor}

\begin{cor}
Let $R$ be a commutative $G$-noetherian $G$-graded ring.
There is an order-preserving bijection between the following sets:
\begin{itemize}
\item
The set of torsion-free classes of $\catmod^{G} R$ closed under grading shifts.
\item
The power set of $\Spec^G R$.
\end{itemize}
\end{cor}

\subsection{The category of quasi-coherent sheaves}
See \cite{GW-I,GW-II} for the basics of scheme theory.
For the monoidal structure on the category of quasi-coherent sheaves,
we refer to \cite[Examples 2.4(b) and 6.3]{HO}, where further references can also be found.

Let $X$ be a scheme.
We denote by $\shO_X$ the structure sheaf of $X$.
For a point $x\in X$,
we denote by $\mm_x$ the maximal ideal of $\shO_{X,x}$ 
and $\kappa(x):=\shO_{X,x}/\mm_x$ the residue field of $x$.

The category $\Qcoh X$ of quasi-coherent $\shO_X$-modules is a Grothendieck cosmos
with the usual tensor product of $\shO_X$-modules.
The dualizable objects of $\Qcoh X$ are precisely the locally free $\shO_X$-modules of finite rank.
The invertible objects of $\Qcoh X$ are precisely the invertible $\shO_X$-modules.
Thus, a subcategory $\catX$ of $\Qcoh X$ is L-closed if and only if, for any $\shF\in \catX$ and any invertible $\shO_X$-module $\shL$, we have $\shF\otimes_{\shO_X} \shL \in \catX$.

Let $X$ be a noetherian scheme.
Then $\Qcoh X$ is locally noetherian, and $(\Qcoh X)_{\noeth}$ is nothing but the category $\coh X$ of coherent $\shO_X$-modules.
Thus, $\Qcoh X$ is a locally noetherian Grothendieck cosmos satisfying (C1).
By \cref{rmk:size of dualizable}, it satisfies (C2) if and only if $X$ has the \emph{resolution property}, that is, every coherent $\shO_X$-module is a quotient of a locally free $\shO_X$-module.

Consider the following conditions for a noetherian scheme $X$:
\begin{enur}
\item
$X$ is quasi-projective over a commutative noetherian ring $R$.
\item
$X$ is divisorial (see \cite{Borelli,SGA6}).
In particular, $\Qcoh X$ is generated by a set of invertible $\shO_X$-modules.
\item
$\Qcoh X$ is generated by a set of invertible-filtered objects.
\item
$X$ has the resolution property.
\end{enur}
Then (i)$\imply$(ii)$\imply$(iii)$\imply$(iv) holds.
Thus, there are sufficiently many noetherian schemes satisfying (iii).
Note that there is a proper algebraic surface satisfying (iv) but not (ii) (see \cite[Example 1.7 and Theorem 5.2]{Gross}).
The author does not know whether (ii) and (iii) are equivalent.
We also note that the resolution property ensures that $X$ has affine diagonal,
that is, the intersection of any two affine open subsets is affine (see \cite[Proposition 1.3]{Totaro}).

Let $X$ be a noetherian scheme such that $\Qcoh X$ is generated by a set of invertible-filtered objects.
Let $\calL$ be a complete set of representatives of the isomorphism classes of invertible $\shO_X$-modules.
Then $\Qcoh X$ satisfies \cref{setup:classify subcat},
and $G_{\calL}$ in \cref{setup:classify subcat} coincides with the Picard group $\Pic X:=\Pic(\Qcoh X)$.
From \cite[Theorem 7.6]{Kan-CatSp},
we have a bijection
\[
X \isoto \Spec(\Qcoh X), \quad
x \mapsto \ol{{j_x}_*\kk(x)},
\]
where $j_x$ is the natural morphism $\Spec \shO_{X,x} \to X$.
Then $\Pic X$ trivially acts on $\Spec(\Qcoh X)$
by the projection formula (cf.\ \cite[Proposition 22.81]{GW-II}).
Thus, we have $\Spec^{\Pic X}(\Qcoh X)=\Spec(\Qcoh X) \iso X$.
By \cite[Corollary 7.7, Proposition 7.12, and Lemma 8.1]{Kan-CatSp},
the condition ($*$) of \cref{prp:Ass^G X controlls X} is satisfied.
Thus, we obtain the following slight generalization of the result for divisorial noetherian schemes of \cite{Sai-coh}.
We note that the argument in \cite{Sai-coh} makes essential use of the divisorial property, and therefore the result cannot be proved at this level of generality by the same argument.

\begin{cor}
Let $X$ be a noetherian scheme such that $\Qcoh X$ is generated by a set of invertible-filtered objects.
Let $\catX$ be an L-closed subcategory of $\coh X$.
\begin{enua}
\item
The following are equivalent:
\begin{itemize}
\item $\catX$ is a Serre subcategory.
\item $\catX$ is a torsion class.
\item $\catX$ is a wide subcategory.
\item $\catX$ is an ICE-closed subcategory.
\item $\catX$ is a CE-closed subcategory.
\end{itemize}
\item
The following are equivalent:
\begin{itemize}
\item $\catX$ is a torsion-free class.
\item $\catX$ is an IKE-closed subcategory.
\item $\catX$ is an IE-closed subcategory.
\end{itemize}
\end{enua}
\end{cor}

\begin{cor}
Let $X$ be a noetherian scheme such that $\Qcoh X$ is generated by a set of invertible-filtered objects.
There is an order-preserving bijection between the following sets:
\begin{itemize}
\item
The set of L-closed torsion-free classes of $\coh X$.
\item
The power set of $X$.
\end{itemize}
\end{cor}

\subsection{The category of DG modules}
For the basics of dg rings and dg modules,
we refer to \cite{Keller} and \cite[\href{https://stacks.math.columbia.edu/tag/09JD}{Tag 09JD}]{SP}.
Let $V$ be a complex over a commutative ring.
We use cohomological grading, so that $d^n \colon V^n \to V^{n+1}$.
The complex $V$ is said to be \emph{bounded} (resp.\ \emph{bounded below}, resp.\ \emph{bounded above})
if $V^n=0$ for $|n|\gg 0$ (resp.\ $n \ll 0$, resp.\ $n\gg 0$).
We write $Z^n(V):=\Ker(d^n)$, $B^n(V):=\Ima(d^{n-1})$, and $H^n(V):=Z^n(V)/B^n(V)$.
We denote by $V[i]$ the $i$-th shift of $V$.
For a chain map $f\colon V \to W$ of complexes,
its mapping cone is denoted by $\Cone(f)$.
There is a canonical exact sequence of complexes
\begin{equation}\label{diag:cone ex seq}
0 \to W \to \Cone(f) \to V[1] \to 0,
\end{equation}
where $W \to \Cone(f)$ is the natural inclusion and $\Cone(f) \to V[1]$ is the natural projection.

Let $A$ be a dg ring.
Then the shifts $M[n]$ of dg $A$-modules and the mapping cones $\Cone(f)$ of homomorphisms of dg $A$-modules are again dg $A$-modules.
The exact sequence \eqref{diag:cone ex seq} is also an exact sequence of dg $A$-modules.
The category $\dgMod A$ of (right) dg $A$-modules is a Grothendieck category with a generating set $\{\Cone(\id_A)[n]\}_{n\in \bbZ}$ (see \cite[Example 4.4]{HG} for example).
Since the functor $(\dgMod A)(A,-)\iso Z^0(-)$ preserves directed colimits,
the dg $A$-module $A$ and its shifts are finitely presented objects of $\dgMod A$.

A dg ring is said to be \emph{right dg noetherian}
if every ascending chain of right dg ideals is stationary.
In this case, we write $\dgmod A:=(\dgMod A)_{\noeth}$.
We define \emph{left dg noetherian} similarly.
If a dg ring is both left and right dg noetherian, we call it \emph{dg noetherian}.
\begin{lem}
A dg ring $A$ is right dg noetherian
if and only if the Grothendieck category $\dgMod A$ is locally noetherian.
\end{lem}
\begin{proof}
If $A$ is right dg noetherian, then $\Cone(\id_A)$ is a noetherian object in $\dgMod A$
by the canonical exact sequence \eqref{diag:cone ex seq} for $\Cone(\id_A)$.
This implies $\dgMod A$ is locally noetherian.
Conversely, suppose that $\dgMod A$ is locally noetherian.
Since $A$ is finitely presented, it is a noetherian object.
This proves $A$ is right dg noetherian.
\end{proof}

A dg ring $A$ is said to be \emph{skew-commutative} if $ab=(-1)^{|a||b|}ba$ for any homogeneous elements $a,b \in A$, where $|a|$ denotes the degree of $a$.
In this case, $\dgMod A$ becomes a symmetric monoidal closed category with the usual tensor product of dg $A$-modules (see \cite[Section 4]{Day} and \cite[Theorem 4.2]{HG}).
Its internal Hom is given by the Hom complex $\dgHom_A(M,N)$ of dg $A$-modules.
Thus $\dgMod A$ is a Grothendieck cosmos.
Since the dg $A$-module $A$ is the unit object of the Grothendieck cosmos $\dgMod A$,
its shifts are dualizable objects in $\dgMod A$.
Thus, the dg $A$-module $\Cone(\id_A)[n]$ is also a dualizable object in $\dgMod A$ 
by the canonical exact sequence \eqref{diag:cone ex seq}.
Therefore, the Grothendieck cosmos $\dgMod A$ satisfies (C1) and (C2).

Let $R$ be a skew-commutative dg noetherian dg ring and put $\calL:=\{R[n]\}_{n \in \bbZ}$.
Then $\dgMod R$ satisfies the assumptions in \cref{setup:classify subcat}.
The additive group $\bbZ$ of integers acts on $\dgMod R$ by shift functors.
This action coincides with an action given 
by the natural surjection $\bbZ \surj G_{\calL}$ to the group defined in \cref{setup:classify subcat}.
A $\bbZ$-closed subcategory is precisely a subcategory closed under shifts,
that is, if $X$ belongs to the subcategory, then so does $X[n]$ for any $n\in \bbZ$.
Thus, we obtain the following result from \cref{prp:Serre=CE}.
\begin{cor}
Let $R$ be a dg noetherian skew-commutative dg ring,
and let $\catX$ be a subcategory of $\dgmod R$ closed under shifts.
\begin{enua}
\item
The following are equivalent:
\begin{itemize}
\item $\catX$ is a Serre subcategory.
\item $\catX$ is a torsion class.
\item $\catX$ is a wide subcategory.
\item $\catX$ is an ICE-closed subcategory.
\item $\catX$ is a CE-closed subcategory.
\end{itemize}
\item
The following are equivalent:
\begin{itemize}
\item $\catX$ is a torsion-free class.
\item $\catX$ is an IKE-closed subcategory.
\item $\catX$ is an IE-closed subcategory.
\end{itemize}
\end{enua}
\end{cor}

In the remainder of this subsection, we determine the atom spectrum of $\dgMod A$ under a certain finiteness condition on a dg ring $A$.
A dg ring $A$ is \emph{non-positive} if $A^{i}=0$ for any $i>0$.
For a non-positive dg ring $A$, we have $Z^0(A)=A^0$ and $H^0(A)=A^0/B^0(A)$.
Thus, for any dg $A$-module $M$ and $n \in \bbZ$,
the differential $d^n_M\colon M^n \to M^{n+1}$ is an $A^0$-linear map.
We identify $H^0(A)$-modules with $A^0$-modules annihilated by $B^0(A)$.
Every $H^0(A)$-module can be viewed as a dg $A$-module concentrated in degree $0$,
and the converse also holds.
Therefore, we identify $H^0(A)$-modules with dg $A$-modules concentrated in degree $0$.

For an $A^0$-module $M$, define a dg $A$-module $D(M)$ as follows:
\begin{itemize}
\item
$D(M)$ is a complex $D(M)=(M \xr{\id_M} M )$ concentrated in degree $0$ and $-1$.
\item
For $a\in A^i$ and $x\in D(M)^j$,
we define the action of $a$ on $x$ by
\[
x\cdot a:=
\begin{cases}
xa & \text{if $i=0$}\\
xd_A(a) & \text{if $i=-1$ and $j=0$}\\
0 & \text{else}, 
\end{cases}
\]
\end{itemize}
It is routine to check that $D(M)$ is indeed a dg $A$-module.
We will describe the atom spectrum $\Spec(\dgMod A)$ using the dg $A$-modules $D(M)$.
\begin{lem}\label{prp:map from D(M)}
Let $A$ be a non-positive dg ring.
For any $A^0$-module $M$ and any dg $A$-module $X$,
there is a bijection
\[
(\dgMod A)(D(M),X)\iso \{u \in (\Mod A^0)(M,X^{-1}) \mid u(M)\cdot A^{<0}=0 \}.
\]
\end{lem}
\begin{proof}
For $f\in (\dgMod A)(D(M),X)$,
it is easy to see that the assignment $f \mapsto f^0$ defines a well-defined injective map
from the left-hand side to the right-hand side.
Conversely, for $u \in (\Mod A^0)(M,X^{-1})$ such that $u(M)\cdot A^{<0}=0$,
define a chain map $f \colon D(M) \to X$ by $f^{-1}:=u$ and $f^0:=d_X^{-1}\circ u$.
We prove that $f$ is an $A$-linear map.
The nontrivial part is $f(m\cdot a)=f(m)\cdot a$ for $m\in D(M)^0$ and $a\in A^{-1}$.
We have
\[
f(m\cdot a) = f(md_A(a)) = u(md_A(a))=u(m)d_A(a),\quad
f(m)\cdot a = d_X(u(m))\cdot a.
\]
Since $u(m)\cdot a =0$, we obtain the following by the graded Leibniz rule:
\[
0=d_X(u(m)\cdot a)=d_X(u(m))\cdot a - u(m)\cdot d_A(a).
\]
This yields the desired equality.
\end{proof}

For a dg $A$-module $M$, define the truncation of $M$ at degree $n$ by
\[
\tau^{\le n} M \colon \quad \cdots \to M^{n-2} \to M^{n-1} \to Z^n(M) \to 0 \to 0 \to \cdots.
\]
Then $\tau^{\le n} M$ is a dg $A$-submodule of $M$.
Setting $T_n := \tau^{\le n} M$,
we obtain a filtration $\{T_n\}_{n\in \bbZ}$ of $M$ by dg $A$-submodules
\begin{equation}\label{eq:2-term filt}
0 \subseteq \cdots \subseteq T_{n-1} \subseteq T_n \subseteq \cdots \subseteq M
\end{equation}
such that
\[
\bigcup_{n\in\bbZ} T_n = M,\quad
\bigcap_{n\in\bbZ} T_n =0, \quad 
T_n/T_{n-1} \iso (B^n(M) \inj Z^n(M)),
\]
where $T_n/T_{n-1}$ is concentrated in degrees $n-1$ and $n$.
A complex $M$ is said to be \emph{$2$-term} if there exists $n\in \bbZ$ such that $M^{i}=0$ for all $i\ne n-1,n$.
In this case, we call $d^{n-1}\colon M^{n-1} \to M^n$ the \emph{main differential},
and we often write $M=(M^{n-1} \to M^n)$.
A complex $M$ concentrated in a single degree $n$ is viewed as a $2$-term complex whose main differential is $0 \to M^n$.

We first compare certain finiteness conditions in $\dgMod A$ and  $\Mod A^0$ for a non-positive dg ring $A$.

\begin{lem}\label{prp:2-term inj noeth}
Let $A$ be a non-positive dg ring.
\begin{enua}
\item
If a $2$-term dg $A$-module $(X\inj Y)$ with injective main differential is noetherian in $\dgMod A$, then both $X$ and $Y$ are noetherian $A^0$-modules.
\item
If a dg $A$-module $M$ is noetherian in $\dgMod A$,
then $M^n$ is a noetherian $A^0$-module for any $n\in \bbZ$.
The converse holds when $M$ is bounded.
\end{enua}
\end{lem}
\begin{proof}
(1):
Suppose that a $2$-term dg $A$-module $(X\inj Y)$ with injective main differential is noetherian in $\dgMod A$.
We may assume that $(X\inj Y)$ is concentrated in degrees $-1$ and $0$.
Then we have an exact sequence $0 \to D(X) \to (X\inj Y) \to (0 \to Y/X) \to 0$ by \cref{prp:map from D(M)}.
Thus, both $D(X)$ and $(0 \to Y/X)$ are noetherian in $\dgMod A$.
From this, we can easily see that both $X$ and $Y/X$ are noetherian $A^0$-modules,
and hence so is $Y$.

(2):
Suppose that a dg $A$-module $M$ is a noetherian object in $\dgMod A$.
Consider the filtration $\{T_n\}_{n\in \bbZ}$ of $M$ in \eqref{eq:2-term filt}.
Then, for any $n\in \bbZ$, the $2$-term complex $T_n/T_{n-1}\iso (B^n(M) \inj Z^n(M))$ is a noetherian object in $\dgMod A$.
From (1), both $B^n(M)$ and $Z^n(M)$ are noetherian $A^0$-modules.
By the exact sequence $0 \to Z^n(M) \to M^n \to B^{n+1}(M) \to 0$ of $A^0$-modules,
we conclude that $M^n$ is a noetherian $A^0$-module for any $n\in \bbZ$.
The final statement is clear.
\end{proof}

\begin{cor}\label{prp:nonpos dg noeth}
Let $A$ be a non-positive dg ring.
If $A$ is right dg noetherian, then $A^0$ is a right noetherian ring and $A^n$ is a finitely generated $A^0$-module for any $n$.
The converse holds when $A$ is bounded.\qed
\end{cor}

Now we give representatives of monoform objects in $\dgMod A$ up to atom-equivalence.
\begin{prp}\label{prp:bdd below atom}
Let $A$ be a non-positive dg ring such that $A^0$ is right noetherian.
If a monoform object of $\dgMod A$ is bounded below, then it is atom-equivalent to $D(H)[n]$ for some monoform object $H$ in $\Mod A^0$ and some $n\in \bbZ$.
\end{prp}
\begin{proof}
Let $X$ be a monoform object of $\dgMod A$ and suppose that it is bounded below.
Consider the filtration $\{T_n\}_{n\in \bbZ}$ of $X$ in \eqref{eq:2-term filt}.
As $X$ is bounded below, $T_n=0$ for some $n \ll 0$.
Thus, there exists a nonzero dg $A$-submodule $Y$ of $X$
such that $Y$ is a $2$-term complex with injective main differential.
Replacing $X$ with a shift of $Y$,
we may assume that $X$ is concentrated in degrees $-1$ and $0$, and that its main differential is injective.
There are two cases: $X^{-1}=0$ and $X^{-1}\ne 0$.
If $X^{-1}\ne 0$,
there is a monoform $A^0$-submodule $H$ of $X^{-1}$ since $A^0$ is right noetherian.
Then the natural inclusion $H \inj X^{-1}$ yields a morphism $D(H) \to X$ of dg $A$-modules by \cref{prp:map from D(M)}.
It is injective as $d^{-1}_X$ is injective,
and thus, $X$ is atom-equivalent to $D(H)$ as desired.

Next, consider the case $X^{-1}=0$.
Note that dg $A$-modules concentrated in degree $0$ are naturally identified with $H^0(A)$-modules.
Since dg $A$-modules concentrated in degree $0$ are closed under subobjects and quotients,
it follows that $X$ is also monoform as a $H^0(A)$-module,
and hence as an $A^0$-module.
Then we have an injective morphism $f\colon X \inj D(X)$ of dg $A$-modules
given by $f^0=\id_X$ and $f^i=0$ for any $i\ne 0$.
Note that this map is indeed $A$-linear since $X\cdot B^0(A)=0$.
This proves that $X$ is atom-equivalent to $D(X)$.
\end{proof}

\begin{lem}\label{prp:D(M) atom}
Let $A$ be a non-positive dg ring.
\begin{enua}
\item
For any monoform $A^0$-module $H$,
the dg $A$-module $D(H)$ is monoform.
\item
Let $H$ and $H'$ be monoform objects in $\Mod A^0$, and let $n, n'\in \bbZ$.
If $D(H)[n]$ and $D(H')[n']$ are atom-equivalent, then $n=n'$ and $H$ is atom-equivalent to $H'$.
\end{enua}
\end{lem}
\begin{proof}
Let $M$ be an $A^0$-module, and let $X$ be a dg $A$-submodule of $D(M)$.
Then $X^{-1}$ and $X^0$ are $A^0$-submodules of $M$ such that $X^{-1} \subseteq X^0$.
The differential $d^{-1}_X$ is given by the natural inclusion $X^{-1} \inj X^0$.
In particular, if $X\ne 0$, then $X^0\ne 0$.

(1)
If there is a nonzero dg $A$-submodule $X$ of $D(H)$ such that $D(H)$ and $D(H)/X$ have a common nonzero subobject $Y$.
Then $H=D(H)^{-1}$ and $(D(H)/X)^{-1}=H/X^{-1}$ have a common subobject $Y^0$,
which is nonzero as $Y$ is a nonzero subobject of $D(H)$.
This contradicts the fact that $H$ is a monoform object of $\Mod A^0$.
Thus, $D(H)$ is a monoform object of $\dgMod A$.

(2)
Suppose that $D(H)[n]$ and $D(H')[n']$ are atom-equivalent.
We may assume that $n \ge n'$ and $n=0$.
There exists a common nonzero dg $A$-submodule $X$ of $D(H)$ and $D(H')[n']$.
Then $X^0$ and $X^{-n'}$ are nonzero.
Since $D(H)$ is concentrated in degree $0$ and $-1$, so is $X$.
Thus, we have $n'=0$ as $-n' \ge 0$.
Moreover, $X^0$ is a common nonzero subobject of $H$ and $H'$,
which proves that $H$ and $H'$ are atom-equivalent.
\end{proof}

\begin{prp}\label{prp:Spec dgMod}
Let $A$ be a bounded non-positive right dg noetherian dg ring.
Then there exists a bijection $\Spec(\dgMod A) \iso \Spec(\Mod A^0) \times \bbZ$.
Here, the bijection from the right-hand side to the left-hand side is given by $(\ol{H},n) \mapsto \ol{D(H)[n]}$.
In particular, we have $\Spec^\bbZ(\dgMod A)=\Spec(\Mod A^0)$.
\end{prp}
\begin{proof}
Let $H$ be a monoform object of $\dgMod A$.
Since $\dgMod A$ is locally noetherian, there is a nonzero noetherian subobject of $H$.
As $A$ is bounded, every finitely generated object is bounded below.
Thus, $H$ is atom-equivalent to a bounded below dg $A$-module.
Then the claim follows from \cref{prp:bdd below atom,prp:D(M) atom,prp:nonpos dg noeth}.
\end{proof}

\begin{cor}
Let $A$ be a bounded non-positive right dg noetherian dg ring.
Suppose that $A^0$ is commutative.
Then there exists a bijection $\Spec^\bbZ(\dgMod A) \iso \Spec A^0$.\qed
\end{cor}

Next, we prove that the bijection $\Spec^\bbZ(\dgMod A) \iso \Spec(\Mod A^0)$ is indeed a homeomorphism.
\begin{lem}\label{prp:filt by D(H) and H}
Let $A$ be a bounded non-positive right dg noetherian dg ring.
Then for any $M\in \dgmod A$,
there exists a filtration of $M$ by dg $A$-submodules
\[
0 = L_0 \subseteq L_1 \subseteq L_2 \subseteq \cdots \subseteq L_m  = M
\]
such that, for each $k\ge 1$, one of the following holds:
\begin{enur}
\item
$L_k/L_{k-1} \iso D(H)[n]$ for some monoform object $H$ in $\Mod A^0$ and $n\in\bbZ$
\item
$L_k/L_{k-1}\iso H'[m]$ for some monoform object $H'$ in $\Mod H^0(A)$ and $m\in \bbZ$.
\end{enur}
\end{lem}
\begin{proof}
As $A$ is bounded, any noetherian dg $A$-module is bounded.
Considering the filtration in \eqref{eq:2-term filt}, it is enough to show that every noetherian $2$-term dg $A$-module $M$ with injective main differential has a desired filtration.
We may assume that $M$ is concentrated in degree $-1$ and $0$.
Both $M^{-1}$ and $H^0(M)=M^0/M^{-1}$ are noetherian $A^0$-modules by \cref{prp:2-term inj noeth}.
We have an exact sequence $0 \to D(M^{-1}) \to M \to H^0(M) \to 0$ by \cref{prp:map from D(M)}.
Because $M^{-1}$ has a filtration $\{X_i\}_i$ of $A^0$-modules such that $X_i/X_{i-1}$ is monoform by \cite[Theorem 2.9]{Kan-Serre},
we obtain the filtration $\{D(X_i)\}_i$ of $D(M^{-1})$ satisfying (i): $D(X_i)/D(X_{i-1})\iso D(X_{i}/X_{i-1})$.
On the other hand, a filtration of $H^0(M)$ given by \cite[Theorem 2.9]{Kan-Serre} yields a filtration of $H^0(M)$ as a dg $A$-module satisfying (ii).
Therefore, we obtain the conclusion.
\end{proof}

\begin{lem}\label{prp:SuppZ of D(H)}
Let $A$ be a bounded non-positive right dg noetherian dg ring.
The following hold under the bijection in \cref{prp:Spec dgMod}.
\begin{enua}
\item
For any monoform object $H$ in $\Mod A^0$,
we have $\Supp^\bbZ_{\dgMod A}(D(H))=\Supp_{\Mod A^0} H$.
\item
For any monoform object $H$ in $\Mod H(A)^0$,
we have $\Supp^\bbZ_{\dgMod A}(H)=\Supp_{\Mod A^0} H$.
\end{enua}
\end{lem}
\begin{proof}
(1)
Let $K$ be a monoform object of $\Mod A^0$.
If $\wti{D(K)}\in \Supp^\bbZ_{\dgMod A}(D(H))$,
then there is a subobject $M$ of $D(H)$ such that $D(H)/M$ and $D(K)[m]$ have a nonzero common subobject $X$ for some $m\in \bbZ$.
Then $X^{-m}$ is a nonzero common subobject of $(D(H)/M)^{-m}$ and $K$.
Thus, we have $m=-1$ or $m=0$ since $D(H)/M$ is concentrated in degree $-1$ and $0$.
Therefore, $X^{-m}$ is a monoform subobject of $H/M^{-m}$ that is atom-equivalent to $K$.
This means that $\ol{K} \in \Supp_{\Mod A^0} H$.
Conversely, suppose that $\ol{K} \in \Supp_{\Mod A^0} H$.
There is a subobject $N$ of $H$ such that $H/N$ and $K$ have a nonzero common subobject $Y$.
Then $D(H)/D(N)$ and $D(K)$ have a nonzero common subobject $D(Y)$.
This means that $\ol{D(K)}\in \Supp_{\dgMod A} (D(H))$.

(2) 
It follows from the fact that $\Mod H^0(A)$ is closed under subobjects and quotients as a subcategory of $\dgMod A$.
\end{proof}

\begin{cor}\label{prp:compare Supp dgMod}
Let $A$ be a bounded non-positive right dg noetherian dg ring.
\begin{enua}
\item
For any $M \in \dgMod A$, we have $\Supp^\bbZ_{\dgMod A}(M)=\Supp_{\Mod A^0} M$,
where, on the right-hand side, $M$ is regarded as an $A^0$-module by forgetting the differential and the grading.
\item
The bijection $\Spec^\bbZ(\dgMod A)=\Spec(\Mod A^0)$ in \cref{prp:Spec dgMod} is a homeomorphism.
\end{enua}
\end{cor}
\begin{proof}
(1) follows from \cref{prp:filt by D(H) and H,prp:SuppZ of D(H)} and \eqref{eq:SuppG AssG formula}.
(2) follows from (1), since the respective open bases correspond to each other. 
\end{proof}

Finally, we compare $\Ass_{\dgMod A} M$ and $\Ass_{\Mod A^0} M$.
In particular, we will see that $\dgMod A$ satisfies the condition ($*$) of \cref{prp:Ass^G X controlls X}.
\begin{prp}\label{prp:compare Ass dgMod}
Let $R$ be a bounded non-positive skew-commutative dg noetherian dg ring.
For any $M\in \dgMod R$,
we have $\Ass^{\bbZ}_{\dgMod R} M = \Ass_{\Mod R^0} M$,
where, on the right-hand side, $M$ is regarded as an $R^0$-module by forgetting the differential and the grading.
\end{prp}
\begin{proof}
Since $\dgMod R$ is locally noetherian,
it is enough to show that the equality holds for noetherian objects in $\dgMod R$.

Take $x \in \Ass^{\bbZ}_{\dgMod R} M$ and suppose that $x=\wti{D(H)}$ for some monoform $R^0$-module $H$.
Then there is a monoform subobject $X$ of $M$ which is also a subobject of $D(H)[n]$.
Replacing $M$ with its shift, we may assume $n=0$.
Since $X$ is a subobject of $D(H)$, it is a $2$-term dg $R$-module with injective main differential.
As $X$ is nonzero, so is $X^{0}$.
Thus, $\ol{X^{0}} = \ol{H}$ as $R^0$-modules, and $\ol{H} \in \Ass_{\Mod R^0} M$.

Conversely, take $\alpha \in \Ass_{\Mod R^0} M$.
Since $M= \bigoplus_{n\in \bbZ} M^n$ as an $R^0$-module,
there is some $n\in \bbZ$ such that $\alpha \in \Ass_{\Mod R^0} M^n$.
Assume that $n$ is the smallest one, which exists as $M$ is bounded below.
There is a monoform $R^0$-submodule $H$ of $M^n$ such that $\ol{H}=\alpha$.
For any $d>0$ and any homogeneous element $a\in R^{-d}$,
consider an $R^0$-linear map $r_a \colon H \to M^{n-d}$ given by $r_a(x):=xa$.
If $r_a$ is injective for some $a$, it contradicts the minimality of $n$.
Thus, $\Ker(r_a)$ is a nonzero $R^0$-submodule of $H$.
Since $R^{<0}$ is finitely generated as an $R^0$-module by \cref{prp:nonpos dg noeth},
let $a_1,\dots, a_k$ be homogeneous generators of $R^{<0}$.
Then $K:=\bigcap_{i=1}^k \Ker(r_{a_i})$ is a nonzero submodule of $H$ by \cite[Proposition 2.6]{Kan-Serre}.
Thus, $K$ is a monoform $R^0$-module of $M^n$ such that $\ol{K}=\alpha$ and $K\cdot R^{<0}=0$.
If $K\cap Z^n(M) \ne 0$, then $K':=K\cap Z^n(M)$ can be viewed as a dg $R$-submodule of $M$,
and $\wti{D(H)}=\wti{D(K')}=\wti{K'} \in \Ass^\bbZ_{\dgMod R}(M)$.
If $K\cap Z^n(M) =0$, then the composite $K \inj M^n \to M^{n+1}$ of the inclusion map and the differential map $d^n$ is injective.
Thus, we obtain an injection $D(K)[-n-1] \inj M$ of dg $R$-modules by \cref{prp:map from D(M)}.
Therefore, we have $\wti{D(H)}=\wti{D(K)} \in \Ass^\bbZ_{\dgMod R}(M)$.
\end{proof}

\begin{cor}
Let $R$ be a bounded non-positive skew-commutative dg noetherian dg ring.
There is an order-preserving bijection between the following sets:
\begin{itemize}
\item
The set of torsion-free classes of $\dgmod R$ closed under shifts.
\item
The power set of $\Spec R^0$.
\end{itemize}
\end{cor}
\begin{proof}
Since $\Mod R^0$ satisfies the condition ($*$) of \cref{prp:Ass^G X controlls X},
so does $\dgMod R$ by \cref{prp:compare Supp dgMod,prp:compare Ass dgMod}.
Thus, we obtain the assertion from \cref{prp:classify G-torf}.
\end{proof}

Finally, we give an example of bounded non-positive skew-commutative dg noetherian dg rings.
\begin{ex}
Let $R$ be a commutative noetherian ring.
\begin{enua}
\item
For a bounded complex $V$ of finitely generated $R$-modules such that $V^i=0$ for $i\ge 0$,
we define a dg ring $R \ltimes V$ as follows:
\begin{itemize}
\item 
As a complex, $R \ltimes V$ is a direct sum $R\oplus V$ of complexes.
\item
The composite is defined by $(a,v)\cdot(b,w):= (ab,vb+wa)$ for $(a,v),(b,w) \in R\oplus V$.
\end{itemize}
Then $R \ltimes V$ is a bounded non-positive skew-commutative dg noetherian dg ring.
\item
Let $M$ be a finitely generated $R$-module, and let $\varphi \colon M \to R$ be an $R$-linear map.
Then the Koszul complex $K(\varphi)$ associated to $\varphi$ is a non-positive skew-commutative dg ring
(cf.\ \cite[\href{https://stacks.math.columbia.edu/tag/0622}{Tag 0622}]{SP}).
As $M$ is finitely generated, the dg ring $K(\varphi)$ is bounded, and it is dg noetherian.
\end{enua}
\end{ex}

\subsection{Examples outside our assumptions}\label{ss:non ex}
In this final subsection, we give examples of Grothendieck cosmoi that lie outside the scope of \cref{setup:classify subcat}.
However, we will see that tensor torsion classes are Serre subcategories in these Grothendieck cosmoi.
These examples indicate that the assertion that tensor torsion classes are Serre subcategories holds in a more general setting.

\subsubsection{Quiver representations}
We refer to \cite{ASS} for the basics of quiver representations.
In what follows,
let $Q$ be a finite acyclic quiver and $R$ a commutative ring.
We denote by $Q_0$ the set of vertices of $Q$ and $Q_1$ the set of arrows of $Q$.
An \emph{$R$-linear representation} of $Q$ is a pair $(\{M_i\}_{i\in Q_0}, \{M_a\}_{a\in Q_1})$
such that $M_i$ is an $R$-module for each $i\in Q_0$ and $M_a$ is an $R$-linear map $M_i \to M_j$ for each arrow $a\colon i\to j$ in $Q_1$.
Let $RQ$ denote the path algebra of $Q$ over $R$.
We identify $RQ$-modules with $R$-linear quiver representations of $Q$.
We denote by $e_i$ the idempotent of $RQ$ corresponding to the vertex $i$.
If $R$ is noetherian, so is $RQ$ as it is a free $R$-module of finite rank.
Thus, $\Mod RQ$ is a locally noetherian Grothendieck category.

For $M, N \in \Mod RQ$, define their \emph{pointwise tensor product} $M\otimes N \in \Mod RQ$ by
\[
(M\otimes N)_i := M_i \otimes_R N_i \; (i\in Q_0),\quad
(M\otimes N)_a := M_a \otimes_R N_a \; (a\in Q_1).
\]
Then $\Mod RQ$ is a symmetric monoidal closed category.
Its unit object $I_{\Mod RQ}$ is given by
\[
(I_{\Mod RQ})_i := R \; (i\in Q_0),\quad
(I_{\Mod RQ})_a := \id_R \; (a\in Q_1).
\]
Its internal Hom $[M,N]\in \Mod RQ$ is given by
\begin{align*}
[M, N]_i := \Hom_{RQ}(M \otimes P(i) , N)  \; (i\in Q_0) ,\quad
[M, N]_a := \Hom_{RQ}(M \otimes (a\cdot-), N)  \; (a\in Q_1).
\end{align*}
Here, $P(i):=e_i RQ$ is the projective $RQ$-module corresponding to the vertex $i$,
and $a\cdot- \colon P(j) \to P(i)$ is the left multiplication by an arrow $a\colon i \to j$.
We refer \cite[Section 5]{Day} for an account of this monoidal structure in a more general setting.
Thus, $\Mod RQ$ is a Grothendieck cosmos satisfying (C1).
However, in general, $\Mod RQ$ does not satisfy (C2), that is, it is not generated by a set of dualizable objects.
\begin{lem}
The dualizable objects of $\Mod RQ$ are precisely the $R$-linear representations $D$ of $Q$
such that $D_i$ is a finitely generated projective $R$-module for any $i\in Q_0$ 
and $D_a$ is an isomorphism for any $a\in Q_1$. 
\end{lem}
\begin{proof}
Let $D$ be a dualizable object of $\Mod RQ$.
As the evaluation functor $(-)_i \colon \Mod RQ \to \Mod R$ is strong monoidal by definition,
the $R$-module $D_i$ is also dualizable, and hence, it is finitely generated projective for any $i\in Q_0$.

Let $E$ be the dual of $D$, and let $\eta \colon I \to D\otimes E$ and $\epsilon \colon E\otimes D \to I$ be morphisms satisfying \eqref{diag:characterize dual}.
Then we have the following commutative diagrams in $\catmod R$ for each arrow $a\colon i \to j$ in $Q$:
\begin{equation}\label{diag:dualizable in Mod RQ}
\begin{tikzcd}[row sep=20pt,column sep=40pt]
D_j \ar[r,"\iso"] & R \otimes_R D_j \ar[r,"\eta_i\otimes \id_{D_j}"] & D_i \otimes_R E_i \otimes_R D_j \ar[rd,"\id_{D_i} \otimes E_a \otimes \id_{D_j}"] & \\
D_i \ar[r,"\iso"] \ar[u,"D_a"] & R \otimes_R D_i \ar[r,"\eta_i\otimes \id_{D_i}"] \ar[u,"\id_R \otimes D_a"] \ar[rd,"\iso"',"(1)"] & D_i \otimes_R E_i \otimes_R D_i \ar[u,"\id_{D_i} \otimes \id_{E_i} \otimes D_a"] \ar[r,"\id_{D_i} \otimes E_a \otimes D_a"] \ar[d,"\id_{D_i} \otimes \epsilon_i"] & D_i \otimes_R E_j \otimes_R D_i \ar[ld,"\id_{D_i}\otimes \epsilon_j","(2)"']\\
& & D_i\otimes_R R & .
\end{tikzcd}
\end{equation}
Here, the commutativity of (1) follows from \eqref{diag:characterize dual}.
The commutativity of (2) follows from the following commutative diagram,
which is obtained by the morphism $\epsilon \colon E \otimes D \to I$ of quiver representations:
\[
\begin{tikzcd}
E_i \otimes_R D_i \ar[r,"\epsilon_i"] \ar[d,"E_a \otimes D_a"] & R \ar[d,equal] \\
E_j \otimes_R D_j \ar[r,"\epsilon_j"] & R. \\
\end{tikzcd}
\]
The morphism $D_a$ is a split monomorphism by the commutative diagram \eqref{diag:dualizable in Mod RQ}.
By the dual argument, we see that $D_a$ is a split epimorphism.
Therefore, the $R$-linear map $D_a\colon D_i \to D_j$ is an isomorphism for any arrows $a\colon i \to j $ in $Q$.
\end{proof}

Even though $\Mod RQ$ does not satisfy (C2) and therefore does not satisfy the assumptions of \cref{setup:classify subcat},
the following still holds.
\begin{prp}
Let $R$ be a commutative noetherian ring.
\begin{enua}
\item
Every tensor torsion class in $\catmod RQ$ is a Serre subcategory.
\item
There is a bijection
\[
\serre_{\otimes}(\catmod RQ) \isoto \prod_{i\in Q_0} \serre(\catmod R),\quad
\catX \mapsto (\catX_i)_{i\in Q_0},
\]
where $\catX_i$ is defined by \eqref{eq:tensor tors is Serre in mod RQ} below.
\end{enua}
\end{prp}
\begin{proof}
We first introduce some notation.
Fix a vertex $i\in Q_0$.
For an $R$-module $M$, define an $RQ$-module $L_i(M)$ by
\[
L_i(M)_j:=
\begin{cases}
M & \text{if $j=i$}\\
0 & \text{if $j\ne i$}
\end{cases}
\; (j\in Q_0),\quad
L_i(M)_a := 0 \; (a\in Q_1).
\]
The assignment $M \mapsto L_i(M)$ defines an exact functor $L_i \colon \Mod R \to \Mod RQ$.

Let $\catX$ be a tensor torsion class of $\catmod RQ$.
Define a subcategory of $\catmod R$ by
\begin{equation}\label{eq:tensor tors is Serre in mod RQ}
\catX_i := \{M \in \catmod R \mid \text{$M=V_i$ for some $V\in \catX$}\}.
\end{equation}
Then, for any $M\in \catmod R$, we have that $M \in \catX_i$ if and only if $L_i(M)\in \catX$
since $V\otimes L_i(R)\iso L_i(V_i)$ for any $V \in \Mod RQ$.
As $L_i$ is an exact functor,
the subcategory $\catX_i$ is also a torsion class of $\catmod R$,
and thus, it is a Serre subcategory by \cref{prp:tors=Serre in mod R}.
We prove $\catX$ is determined by $(\catX_i)_{i\in Q_0}$ and it is a Serre subcategory of $\catmod RQ$.
It is enough to show that, for any $V \in \catmod RQ$,
we have that $V\in \catX$ if and only if $V_i \in \catX_i$ for any $i\in Q_0$.
The ``only if'' direction is clear.
Suppose that $V_i \in \catX_i$ for any $i\in Q_0$.
We prove $V \in \catX$ by induction on the number $m$ of vertices $i\in Q_0$ such that $V_i \ne 0$.
If $m=1$, then $V=L_i(V_i)$ for some $i \in Q_0$. Thus, we have $V\in \catX$.
Assume that $m>1$.
Let $i_0 \in Q_0$ be a sink, that is, there are no arrows starting from $i_0$.
Then $L_i(V_i)$ is a $RQ$-submodule of $V$.
By the induction hypothesis, both $V/L_i(V_i)$ and $L_i(V_i)$ belong to $\catX$,
and thus so does $V$.

To prove the bijection (2), consider the following subcategory for any $(\catX_i)_{i\in Q_0} \in \prod_{i\in Q_0} \serre(\catmod R)$:
\[
\{V \in \catmod RQ \mid \text{$V_i \in \catX_i$ for any $i \in Q_0$}\}.
\]
Then it is a tensor Serre subcategory as every Serre subcategory of $\catmod R$ is a tensor ideal (cf.\ \cref{prp:criterion tensor CE}).
From this, we can easily see that the map in (2) is surjective.
The injectivity follows from the first part. 
\end{proof}

\subsubsection{Modules over Hopf algebras}
See \cite[Chapter III]{Kassel} for the basics of Hopf algebras.
A \emph{bialgebra} is an algebra $B$ over a field $\bbk$
equipped with algebra homomorphisms $\Delta \colon B \to B\otimes_\bbk B$ and $\epsilon \colon B \to \bbk$ satisfying certain conditions.
The homomorphisms $\Delta$ and $\epsilon$ are called the \emph{coproduct} and the \emph{counit}, respectively.
For any $M, N \in \Mod B$,
the tensor product $M\otimes_\bbk N$ over $\bbk$ can be regarded as
a $B$-module via the coproduct $\Delta$.
The vector space $\bbk$ can be regarded as a $B$-module via the counit $\epsilon$.
Then $\Mod B$ is a monoidal category with tensor product $\otimes:= \otimes_\bbk$ and unit object $I_{\Mod B}=\bbk$.
Let $\tau\colon B \otimes_\bbk B \isoto B \otimes_\bbk B$ be a $\bbk$-linear isomorphism defined by $\tau(a \otimes b)=b\otimes a$.
The bialgebra $B$ is said to be \emph{cocommutative} if $\tau \circ \Delta = \Delta$.
In this case, $\Mod B$ is a Grothendieck cosmos (cf.\ \cref{prp:criterion for Groth cosmos}).

A \emph{Hopf algebra} is a bialgebra $H$ with an \emph{antipode}, 
that is, an algebra homomorphism $S\colon H \to H^{\op}$ satisfying certain conditions.
For any $M,N \in \Mod H$,
the $\bbk$-vector space $\Hom_{\bbk}(M,N)$ can be viewed as an $H$-module via the antipode $S$
(cf.\ \cite[Section III.5]{Kassel}).
If $H$ is a cocommutative Hopf algebra,
then the internal Hom in $\Mod H$ is given by $\Hom_{\bbk}(M,N)$.
The dualizable objects of $\Mod H$ are precisely the $H$-modules that are finite dimensional over $\bbk$.
Therefore, if $H$ is a finite-dimensional cocommutative Hopf algebra,
then $\Mod H$ is a locally noetherian Grothendieck cosmos satisfying (C1) and (C2).
However, in general, it does not satisfy the assumptions in \cref{setup:classify subcat}.
Nevertheless, we obtain the following observation.
\begin{prp}
Let $H$ be a cocommutative Hopf algebra that is right noetherian as a ring.
Then every tensor torsion class of $\catmod H$ is either zero or $\catmod H$.
In particular, every tensor torsion class is a Serre subcategory.
\end{prp}
\begin{proof}
Let $\catX$ be a nonzero tensor torsion class of $\catmod H$.
Take a nonzero object $X\in\catX$.
Since the evaluation map $\Hom_{\bbk}(X,\bbk)\otimes_\bbk X \to \bbk$ is a nonzero $H$-linear map, it is surjective.
Since $\catX$ is a tensor torsion class,
we have $\bbk \in \catX$, and hence $\catX=\catmod H$.
\end{proof}

Finally, we give examples of noetherian cocommutative Hopf algebras.
\begin{ex}
Let $\bbk$ be a field.
\begin{enua}
\item
For a finite group $G$,
the group algebra $\bbk G$ is a finite-dimensional cocommutative Hopf algebra
whose coproduct, counit, and antipode are given by
\[
\Delta(g)=g\otimes g,\quad
\epsilon(g)=1,\quad
S(g)=g^{-1}\quad
(g\in G).
\]
\item
For a finite-dimensional Lie algebra $\mathfrak{g}$,
the universal enveloping algebra $U(\mathfrak{g})$
is a noetherian cocommutative Hopf algebra
whose coproduct, counit, and antipode are given by
\[
\Delta(x)=x\otimes 1 + 1\otimes x,\quad
\epsilon(x)=0,\quad
S(x)=-x\quad
(x\in \mathfrak{g}).
\]
For the fact that $U(\mathfrak{g})$ is noetherian, we refer to \cite[Chapter 1, \S 7.4]{MR}, for example.
\end{enua}
\end{ex}


\begin{thebibliography}{NNNN00}
\bibitem[AIR14]{AIR}
 T.\ Adachi, O.\ Iyama, I.\ Reiten,
 \emph{$\tau$-tilting theory},
 Compos.\ Math.\ \textbf{150} (2014), no.\ 3, 415--452.

\bibitem[Asa20]{Asai}
 S.\ Asai,
 \emph{Semibricks},
 Int.\ Math.\ Res.\ Not.\ IMRN (2020), no.\ 16, 4993--5054.

\bibitem[AS16]{AS}
 K.\ Ahmadi, R.\ Sazeedeh,
 \emph{Hereditary torsion theories of a locally Noetherian Grothendieck category},
 Bull.\ Aust.\ Math.\ Soc.\ \textbf{94} (2016), no.\ 3, 421--430.

\bibitem[ASS06]{ASS}
 I.\ Assem, D.\ Simson, A.\ Skowro\'{n}ski,
 \emph{Elements of the representation theory of associative algebras. Vol.\ 1.
 Techniques of representation theory},
 London Math.\ Soc.\ Stud.\ Texts, \textbf{65}
 Cambridge University Press, Cambridge, 2006. x+458 pp.

\bibitem[Bak20]{Bak}
 R.\ H.\ Bak,
 \emph{Dualizable and semi-flat objects in abstract module categories},
 Math.\ Z.\ \textbf{296} (2020), no.\ 1--2, 353--371.

\bibitem[Bal05]{Balmer}
 P.\ Balmer,
 \emph{The spectrum of prime ideals in tensor triangulated categories},
 J.\ Reine Angew.\ Math.\ \textbf{588} (2005), 149--168.

\bibitem[BKS19]{BKS}
 P.\ Balmer, H.\ Krause, G.\ Stevenson,
 \emph{Tensor-triangular fields: ruminations},
 Selecta Math.\ \textbf{25} (2019), no.\ 1, Paper No.\ 13, 36 pp.

\bibitem[Bor63]{Borelli}
 M.\ Borelli,
 \emph{Divisorial varieties},
 Pacific J.\ Math.\ \textbf{13} (1963), 375--388.

%

\bibitem[\v{C}\v{S}20]{CS}
 P.\ \v{C}oupek, J.\ \v{S}{\v t}ov\'{i}\v{c}ek,
 \emph{Cotilting sheaves on Noetherian schemes},
 Math.\ Z.\ \textbf{296} (2020), no.\ 1--2, 275--312.

\bibitem[Day70]{Day}
 B.\ Day,
 \emph{On closed categories of functors},
 Reports of the Midwest Category Seminar, IV, pp.\ 1--38,
 Lecture Notes in Math., Vol.\ 137
 Springer-Verlag, Berlin-New York, 1970.

\bibitem[Eno22]{Eno-ICE}
 H.\ Enomoto,
 \emph{Rigid modules and ICE-closed subcategories in quiver representations},
 J.\ Algebra \textbf{594} (2022), 364--388.

\bibitem[Eno]{Eno}
 H.\ Enomoto,
 \emph{IE-closed subcategories of commutative rings are torsion-free classes},
 preprint (2023), arXiv:2304.03260v2.

\bibitem[ES22]{ES-ICE}
 H.\ Enomoto, A.\ Sakai,
 \emph{ICE-closed subcategories and wide $\tau$-tilting modules},
 Math.\ Z.\ \textbf{300} (2022), no.\ 1, 541--577.

\bibitem[EGNO15]{EGNO}
 P.\ Etingof, S.\ Gelaki, D.\ Nikshych, V.\ Ostrik,
 \emph{Tensor categories},
 Math.\ Surveys Monogr.\ \textbf{205},
 American Mathematical Society, Providence, RI, 2015, xvi+343 pp.

\bibitem[Gab62]{Gab}
 P.\ Gabriel,
 \emph{Des cat\'{e}gories ab\'{e}liennes},
 Bull.\ Soc.\ Math.\ France \textbf{90} (1962), 323--448.

\bibitem[GW20]{GW-I}
 U.\ G\"{o}rtz, T.\ Wedhorn,
 \emph{Algebraic geometry I. Schemes--with examples and exercises}, Second edition,
 Springer Studium Mathematik--Master,
 Springer Spektrum, Wiesbaden, 2020, vii+625 pp.

\bibitem[GW23]{GW-II}
  U.\ G\"{o}rtz, T.\ Wedhorn,
 \emph{Algebraic geometry II: Cohomology of scheme--with examples and exercises},
 Springer Stud.\ Math.\ Master,
 Springer Spektrum, Wiesbaden, 2023, vii+869 pp.

\bibitem[Gro12]{Gross}
 P.\ Gross,
 \emph{The resolution property of algebraic surfaces},
 Compos.\ Math.\ \textbf{148} (2012), no.\ 1, 209--226.

\bibitem[Her97]{Herzog-Spec}
 I.\ Herzog,
 \emph{The Ziegler spectrum of a locally coherent Grothendieck category},
 Proc.\ London Math.\ Soc.\ (3) \textbf{74} (1997), no.\ 3, 503--558.

\bibitem[HO23]{HO}
 H.\ Holm, S.\ Odaba\c{s}\i,
 \emph{The tensor embedding for a Grothendieck cosmos}.
 Sci.\ China Math.\ \textbf{66} (2023), no.\ 11, 2471--2494.

\bibitem[HG16]{HG}
 H.\ Al Hwaeer, G.\ Garkusha,
 \emph{Grothendieck categories of enriched functors},
 J.\ Algebra \textbf{450} (2016), 204--241.

\bibitem[IT09]{IT}
 C.\ Ingalls, H.\ Thomas,
 \emph{Noncrossing partitions and representations of quivers},
 Compos.\ Math.\ \textbf{145} (2009), no.\ 6, 1533--1562.

\bibitem[IMST24]{IMST}
 K.\ Iima, H.\ Matsui, K.\ Shimada, R.\ Takahashi,
 \emph{When Is a Subcategory Serre or Torsion-Free?},
 Publ.\ Res.\ Inst.\ Math.\ Sci.\ \textbf{60} (2024), no.\ 4, 831--857.

\bibitem[IK24]{IK}
 O.\ Iyama, Y.\ Kimura,
 \emph{Classifying subcategories of modules over Noetherian algebras},
 Adv.\ Math.\ \textbf{446} (2024), Paper No.\ 109631.

\bibitem[JLV93]{JLV}
 A.\ Jerem\'{i}as L\'{o}pez, M.\ P.\ L\'{o}pez, E.\ Villanueva N\'{o}voa,
 \emph{Localization in symmetric closed Grothendieck categories},
 Third Week on Algebra and Algebraic Geometry (SAGA III),
 Bull.\ Soc.\ Math.\ Belg.\ S\'{e}r.\ A \textbf{45} (1993), no.\ 1--2, 197--221.

\bibitem[Kan12]{Kan-Serre}
 R.\ Kanda,
 \emph{Classifying Serre subcategories via atom spectrum},
 Adv.\ Math.\ \textbf{231} (2012), no.\ 3--4, 1572--1588.

\bibitem[Kan15a]{Kan-spcl}
 R.\ Kanda,
 \emph{Specialization orders on atom spectra of Grothendieck categories},
 J.\ Pure Appl.\ Algebra \textbf{219} (2015), no.\ 11, 4907--4952.

\bibitem[Kan15b]{Kan-CatSp}
 R.\ Kanda,
 \emph{Classification of categorical subspaces of locally Noetherian schemes},
 Doc.\ Math.\ \textbf{20} (2015), 1403--1465.

\bibitem[Kas95]{Kassel}
 C.\ Kassel,
 \emph{Quantum groups},
 Grad.\ Texts in Math., \textbf{155},
 Springer-Verlag, New York, 1995. xii+531 pp.

\bibitem[Kel82]{Kelly}
 G.\ M.\ Kelly,
 \emph{Basic concepts of enriched category theory},
 London Math.\ Soc.\ Lecture Note Ser.\ \textbf{64},
 Cambridge University Press, Cambridge-New York, 1982. 245 pp.

\bibitem[Kel94]{Keller}
 B.\ Keller,
 \emph{Deriving {DG} categories},
 Ann.\ Sci.\ \'{E}cole Norm.\ Sup.\ (4) \textbf{27} (1994), no.\ 1, 63--102.

\bibitem[KS24]{KS}
  T.\ Kobayashi, S.\ Saito,
 \emph{When are KE-closed subcategories torsion-free classes?},
 Math.\ Z.\ \textbf{307} (2024), no.\ 4, Paper No.\ 65.

\bibitem[KS]{KS-2}
 T.\ Kobayashi, S.\ Saito,
 \emph{Classifying KE-closed subcategories over a commutative noetherian ring},
 preprint (2025), arXiv:2509.05767

\bibitem[Kra97]{Krause-Spec}
 H.\ Krause,
 \emph{The spectrum of a locally coherent category},
 J.\ Pure Appl.\ Algebra \textbf{114} (1997), no.\ 3, 259--271.

\bibitem[Kra08]{Krause-torf}
 H.\ Krause,
 \emph{Thick subcategories of modules over commutative Noetherian rings (with an appendix by Srikanth Iyengar)},
 Math.\ Ann.\ \textbf{340} (2008), no.\ 4, 733--747.

\bibitem[Kra15]{Kra-KS}
 H.\ Krause,
 \emph{Krull--Schmidt categories and projective covers},
 Expo.\ Math.\ \textbf{33} (2015), no.\ 4, 535--549.

\bibitem[Mac98]{CWM}
 S.\ Mac Lane,
 \emph{Categories for the working mathematician},
 Second edition.
 Grad.\ Texts in Math.\ \textbf{5},
 Springer-Verlag, New York, 1998. xii+314 pp.

\bibitem[MS17]{MS}
 F.\ Marks, J.\ \v{S}t'ov\'{i}\v{c}ek,
 \emph{Torsion classes, wide subcategories and localisations},
 Bull.\ Lond.\ Math.\ Soc.\ \textbf{49} (2017), no.\ 3, 405--416.

\bibitem[MR01]{MR}
 J.\ C.\ McConnell, J.\ C.\ Robson,
 \emph{Noncommutative Noetherian rings},
 With the cooperation of L.\ W.\ Small.\ Revised edition,
 Grad.\ Stud.\ Math., \textbf{30},
 American Mathematical Society, Providence, RI, 2001. xx+636 pp.

\bibitem[NO04]{NO}
 C.\ N\v{a}st\v{a}sescu, F.\ Van Oystaeyen,
 \emph{Methods of graded rings},
 Lecture Notes in Math.\ \textbf{1836}
 Springer-Verlag, Berlin, 2004. xiv+304 pp.

\bibitem[Pos19]{Posur}
 S.\ Posur,
 \emph{Atom spectra of graded rings and sheafification in toric geometry},
 J.\ Algebra \textbf{534} (2019), 207--227.

\bibitem[Sai25]{Sai-coh}
  S.\ Saito,
 \emph{Classifying torsionfree classes of the category of coherent sheaves and their Serre subcategories},
 J.\ Pure Appl.\ Algebra \textbf{229} (2025), no.\ 1, Paper No.\ 107799, 34 pp.

\bibitem[SW11]{SW}
 D.\ Stanley, B.\ Wang,
 \emph{Classifying subcategories of finitely generated modules over a Noetherian ring},
 J.\ Pure Appl.\ Algebra \textbf{215} (2011), no.\ 11, 2684--2693.

\bibitem[Ste75]{Ste}
 B.\ Stenstr\"{o}m,
 \emph{Rings of quotients. An introduction to methods of ring theory},
 Die Grundlehren der mathematischen Wissenschaften, Band \textbf{217}
 Springer-Verlag, New York-Heidelberg, 1975. viii+309 pp.

\bibitem[Tak08]{Takahashi}
 R.\ Takahashi,
 \emph{Classifying subcategories of modules over a commutative Noetherian ring},
 J.\ Lond.\ Math.\ Soc.\ (2) \textbf{78} (2008), no.\ 3, 767--782.

\bibitem[Tot04]{Totaro}
 B.\ Totaro,
 \emph{The resolution property for schemes and stacks},
 J.\ Reine Angew.\ Math.\ \textbf{577} (2004), 1--22.

\bibitem[SGA6]{SGA6}
  P.\ Berthelot, A.\ Grothendieck, L.\ Illusie,
 \emph{Th\'{e}orie des intersections et th\'{e}or\`{e}me de Riemann-Roch},
 Lecture Notes in Mathematics, \textbf{225},
 Springer-Verlag, 1971.

\bibitem[SP]{SP}
 The {Stacks Project Authors}, 
 \emph{Stacks Project},
 \url{https://stacks.math.columbia.edu}.
\end{thebibliography}
\end{document}